\documentclass[reqno,english]{amsart}
\usepackage{amsfonts,amsmath,latexsym,verbatim,amscd,mathrsfs,color,array}
\usepackage[colorlinks=true]{hyperref}

\usepackage{amssymb}
\usepackage{pdfsync}
\usepackage{epstopdf}

\newcommand{\ass}{\quad\mbox{as}\quad}

\newcommand{\EE}{{\mathcal E}  }

\newcommand{\tg}{{\tt g}}
\newcommand{\inn}{{\quad\hbox{in } }}
\newcommand{\onn}{{\quad\hbox{on } }}
\newcommand{\ttt}{\tilde }

\newcommand{\TT}{{\mathcal T}  }
\newcommand{\LLL}{{\mathcal L}  }

\newcommand{\nn}{ {\nabla}  }

\newcommand{\py}{ {\tt y } }
\newcommand{\pz}{ {\tt z } }

\newcommand{\pp}{ {\partial} }

\newcommand{\vp}{\varphi}

\newcommand{\OO}{{\mathcal O}}

\newcommand{\NN}{ {\mathcal N}}
\newcommand{\FF}{ {\mathcal F}}
\newcommand{\II}{ {\mathbb I}}
\newcommand{\N}{\mathbb{N}}

\newcommand{\R} {\mathbb R}

\newcommand{\cuad}{{\sqcap\kern-.68em\sqcup}}

\newcommand{\DD}{{\mathcal D}}

\newcommand{\BB }{{\mathcal B}}

\renewcommand{\a}{{\alpha}}
\renewcommand{\b}{{\beta}}

\newcommand{\foral}{\quad\mbox{for all}\quad}
\newcommand{\ve}{\varepsilon}

\newcommand{\be}{\begin{equation}}
\newcommand{\ee}{\end{equation}}

\newcommand{\la}{\lambda}

\newcommand{\equ}[1]{(\ref{#1})}

\newcommand{\omeganr} {\omega^{nr}}
\newcommand{\omegazero} {\omega^0}
\newcommand{\gf}{{g}}		
\newcommand{\metric}[1]{{{\langle #1\rangle}_m}}

\newtheorem{definition}{Definition}
\newtheorem{lemma}{Lemma}[section]
\newtheorem{prop}{Proposition}[section]
\newtheorem{theorem}{Theorem}

\newtheorem{remark}{Remark}[section]
\newcommand{\bremark}{\begin{remark} \em}
\newcommand{\eremark}{\end{remark} }

\numberwithin{equation}{section}

\begin{document}

\title[Interface dynamics in semilinear wave equations]{Interface dynamics in semilinear wave equations}


\author[M. del Pino]{Manuel del Pino}
\address{\noindent   Department of Mathematical Sciences University of Bath,
Bath BA2 7AY, United Kingdom \\
and  Departamento de
Ingenier\'{\i}a  Matem\'atica-CMM   Universidad de Chile,
Santiago 837-0456, Chile}
\email{m.delpino@bath.ac.uk}

\author[R. L. Jerrard]{Robert L. Jerrard}
\address{\noindent
Department of Mathematics,
University of Toronto, Toronto, Ontario, M5S 2E4, Canada}
\email{rjerrard@math.toronto.edu}

\author[M. Musso]{Monica Musso}
\address{\noindent   Department of Mathematical Sciences University of Bath,
Bath BA2 7AY, United Kingdom \\
and Departamento de Matem\'aticas, Universidad Cat\'olica de Chile, Macul 782-0436, Chile}
\email{m.musso@bath.ac.uk}

\begin{abstract}
We consider the wave equation $\ve^2(-\pp_t^2  + \Delta)u + f(u) = 0$ for $0<\ve\ll 1$, where
$f$ is the derivative of a balanced, double-well potential, the model case being
$f(u) = u-u^3$.  For equations of this form, we construct solutions that exhibit
an interface of  thickness $O(\ve)$ that
separates regions where the solution is $O(\ve^k)$ close to $\pm 1$, and that
is close to a timelike hypersurface of vanishing {\em Minkowskian} mean curvature.
This provides a Minkowskian analog of the numerous
results that connect the Euclidean Allen-Cahn equation and minimal surfaces
or the parabolic Allen-Cahn equation and motion by mean curvature.
Compared to earlier results of the same character, we develop a new
constructive approach that applies to a larger class of nonlinearities and
yields much more precise information about the solutions under consideration.

\end{abstract}

\maketitle

\section{Introduction}

Consider the initial value problem
\be\label{1}\left\{
\begin{aligned}
\ve^2 \Box u \, +\,  f(u )\ = &\ 0 \inn [0,T]\times \R^{n}\\
u = u_0, \quad  \pp_t u = &u_1\inn \{0\}\times \R^n
\end{aligned}\right.
\ee
where
$
\Box u  =  -\pp^2_t u   + \Delta_x u$ and  $\Delta_x u  =\sum_{i=1}^n \pp^2_{x_i} u.
$
We are interested in nonlinearities of the form
$$
f(s) = - W'(s)
$$
where  $W(s)$ is a ``balanced double-well potential'', namely a $C^\infty$ even function such that
\be\label{double} \begin{aligned} W(s)\,> & \, 0 \inn \R \setminus\{-1,1\}\,\\    W(\pm 1)\, = &\, W'(\pm 1)\, =\, 0, \\ W''(\pm 1)\, = &\,  a >0 \end{aligned} \ee
A canonical example is the wave version of the Allen-Cahn equation $$W(u) = \frac 14 (1-u^2)^2$$ sometimes called the $\phi^4$ model.

\medskip

Since the mid 70's, 
it has been accepted in the physics and cosmology literature (see for example \cite{kibble, NO, vs}) that under some circumstances, solutions of \eqref{1}
should exhibit an interface, separating
regions where $u\approx 1$ and $u\approx -1$, that approximately sweeps out
a timelike minimal surface in Minkowski space. (The timelike Minkowskian minimal surface
equation --- the condition that the mean curvature, with respect to the Minkowski
metric, vanishes identically --- is a quasilinear geometric wave equation
described below.)
Formal asymptotic arguments in support of the same picture
have been known in the applied mathematics literature for about 20 years,
see for example \cite{neu, rn}.
The first rigorous verification of this scenario appeared only
in rather recent work of the
second author and collaborators \cite{J1, galjer, es-j},
which  constructed solutions of \eqref{1}
with an interface concentrated near a timelike minimal surface.

\medskip

In this paper, we revisit this problem,
developing an entirely new approach that
yields stronger results and is likely to be more robust and flexible.
In doing so, we are largely motivated by the clear analogy between
the problem we study and the numerous
classical results concerning solutions of
the elliptic Allen-Cahn equation
\[
\ve^2\Delta u + f(u) = 0  \qquad\mbox{ in }\Omega\subset \R^n
\]
with interfaces that concentrate near (Euclidean) minimal hypersurfaces in $\Omega$.
Many proofs in the elliptic setting
fall into one of two large families:
\begin{itemize}
\item proofs involving $\Gamma$-convergence
and related ideas, see for example \cite{modica, sternberg},
which proceed by characterizing energy concentration when $0<\ve\ll1$, and
\item proofs involving Liapunov-Schmidt reduction or its variants,
ultimately relying on a linearization of the equation about
an approximate solution built around a
minimal surface.
\end{itemize}
The latter family of arguments has a number of advantages over
the former --- it is capable of providing much more precise
descriptions of the solutions being studied \cite{Pacard};
it is more readily adapted to studying solutions of finite (nonzero) Morse index;
it can be used to build entire solutions of the
Allen-Cahn equations \cite{dkw1}, including counterexamples to the de Giorgi conjecture \cite{dkw2};
it can be used to study refined phenomena such as interface
foliation \cite{dkpw,agudelo}.   


\medskip

Prior rigorous work on timelike minimal surfaces and interfaces
in solutions of \eqref{1} is more similar in spirit to the first family of elliptic
results described above --- all papers to date rely on weighted energy estimates to
show that under suitable hypotheses, energy concentrates near a timelike minimal
surface.
In this paper, by contrast,
we  aim to adapt to the hyperbolic
setting techniques from the second family of elliptic results
--- for example, linearization about a high-order  approximate solution.
Thus, our proofs may be loosely seen as hyperbolic analogs of those
in \cite{Pacard, dkw1, dkw2}.
Our results show that as with elliptic problems,
this approach yields a much sharper description of the solutions
constructed than appears to be available from energy estimates
alone. A more detailed
comparison of our results with earlier work is given
in Section \ref{S:comp}. Interfaces in the parabolic analog of Equation \equ{1},
$$ -\pp_t u  + \Delta u  +  \ve^{-2} f(u)  =  0 \inn \R^N\times (0,T) $$
located near solutions of mean curvature flow for surfaces is also a subject that has been widely treated. See
\cite{deMottoni,xc,ess,ilm}. See also \cite{gkikas} and references therein for the corresponding interface foliation problem.



\subsection{Some preliminaries}

Let  $J$ denote the $(n+1)\times (n+1)$ diagonal matrix
$$
J\ :=\  \left [ \begin{matrix} -1  & 0 & \cdots & \cdots &  \cdots & 0  \\
                  0    &   1 & 0 &\cdots    &   \cdots &    0 \\
                  0& 0 & 1 & 0 & \cdots &0 \\
                  \vdots &    \cdots   & \cdots &  \ddots  & 1& \vdots \\
                  0  &     \cdots & \cdots & \cdots &  0& 1  \end{matrix} \right ].
$$
We consider the standard  Minkowski inner product
$$
\langle  a , b  \rangle_m =  a\cdot J b, \quad a,b \in \R^{n+1},
$$
where $\cdot$ denotes the standard Euclidean inner product.

We let $\nu$ be a Minkowki unit normal vector field along $\Gamma$.
This means that $\nu$ satisfies $|\langle \nu,\nu\rangle_m |=1$ and
$\langle \nu, \tau\rangle_m = 0$ for all vectors $\tau$ tangent to $\Gamma$.
Since $J^2$ is the identity, it is easy to see that $\nu = J \bar \nu$, where $\bar \nu$ is normal to $\Gamma$ with respect to the
Euclidean inner product.

An orientable hypersurface $\Gamma$ in $\R^{n+1}$ with
Minkowski normal vector field $\nu$
is said to be  {\bf time-like} if $\metric{ \nu, \nu}  >0  \onn \Gamma $.
Normalizing $\nu$ we will always assume
$$
\metric{  \nu  , \nu  } = 1.
$$

A basic fact is that under assumptions \equ{double} there is a unique solution to the problem
\be\label{homo}
w''(\zeta)  +  f(w(\zeta)), \quad w(0)= 0,\quad  w(\pm \infty) = \pm 1.
\ee
which is defined by the relation
$$
\zeta =  \int_0^{w(\zeta)} \frac{ ds}{\sqrt{ 2 W(s)}} .
$$
$w(\zeta)$ is an odd function since $W$ is even. It satisfies
$$ w(\zeta)\to \pm 1 + O(e^{-a|\zeta|}) \ass \zeta \to \pm \infty ,
$$
and
\[
D_\zeta^kw(\zeta)= O(e^{-a|\zeta|}) \ass \zeta \to \pm\infty \qquad\mbox{ for all }k\in \N.
\]
In the case of the Allen Cahn nonlinearity $f(u)= u(1-u^2)$, we explicitly have
$$
w(\zeta)  \ =\ \tanh \left (\frac{\zeta}{\sqrt{2}}  \right ) .
$$
We will need a standard fact for the quadratic form associated to the linearization of equation \equ{homo}:  there is a
positive constant $c$ such that
for any $\psi \in H^1(\R)$ with $\int_{\R}  \psi w' = 0 $ we have
\be \label{quadratic}
Q(\psi) :=  \int_\R  |\psi'|^2 - f'(w)\psi^2 \ \ge \ c \int_\R |\psi'|^2  + |\psi|^2 .
\ee
 This estimate follows from a direct compactness argument and the identity
$$Q(\psi) =   \int_\R  {w'}^2 |\rho '|^2, \quad \psi = \rho w' . $$

\medskip
Using $w(\zeta)$ we can find a class of explicit examples to solutions with phase transition across {\em time-like planes}.
Let $\Gamma $ be a time-like hyperplane
in $\R^{n+1}$ with Minkowski normal $\nu$. Then
for any $p\in \Gamma$,  all  its points can be described as the set
\be\label{hyp}
\Gamma =  \{Y= (x,t)\in \R^{n+1} \ /\   \big\langle\, Y-p \, ,\, \nu  \,\big\rangle_m \,=\, 0 \, \}, \quad \big\langle \nu , \nu  \big\rangle=1.
\ee

We observe that all points $(x,t)\in \R^{n+1}$ can be expressed as
$$
(x,t) \ =\ p  + z\nu , \quad  (p,z)\in \Gamma\times  \R .
$$
Clearly we have
$ z=  \big\langle  (x,t) -p,\nu \big\rangle_m $. Let us consider the function
\be \label{assi} u(x,t) = w \left ( \frac z\ve  \right ), \quad   z =  \big \langle (x,t) - p, \nu \big \rangle_m \ee
We quickly check that
$$
\ve^2 {{\Box} }u(x,t) + f(u(x,t))  =   \langle \nu,  \nu   \rangle_m \, w''(\zeta)  + f( w(\zeta)) = 0, \quad \zeta = \frac z\ve,
$$
and hence $u$ solves \equ{1} with a sharp transition on $\Gamma$ between the values $-1$ and $+1$,  for suitable initial data.


\medskip

\subsection{Statement of the Main Result}\label{Statement}

Next we introduce the objects and notation necessary for the statement of our main result.

\medskip
$\bullet$  We assume that $\Gamma$  is a smooth, {\bf time-like hypersurface} in $[0,T]\times \R^{n}$   that divides
the space
$[0,T] \times \R^n$  into two disjoint open components $\OO^-$ and $\OO^+$ with $\OO^-$ being bounded.

\medskip
$\bullet$
We assume in addition that $\Gamma$ is a {\bf Minkowski minimal surface} in $\R^{n+1}$ (in the sense of Definition \ref{minimal} below), and that the velocity of $\Gamma$ vanishes at $t=0$.
We remark that the Cauchy problem for timelike minimal surfaces is studied in
\cite{Brendle, lind, mil}.

\medskip
$\bullet$ We also assume that there exists some $\delta>0$ such that
\be
(Y,z)\in \Gamma\times (-\delta, \delta) \ \ \longrightarrow \ \  (x,t) =  Y+z\nu(Y) \quad\mbox{ is injective},
\label{fer} \ee
where $\nu(Y)$ is a Minkowski normal vector field on $\Gamma$ with $$\langle \nu(Y), \nu(Y) \rangle_m =1. $$

Let us call $\NN$ the set of all points of the form \equ{fer}. For a function $\xi(x,t) $ defined on $\NN$ sufficiently smooth we write
$ D^j_Y D^l_z \xi(Y,z) $ meaning iterated directional derivatives respectively on tangent directions to $\Gamma$ at $Y$ or in $\nu$-direction.
We choose $\nu$ to be the normal pointing towards
$\OO^+$.
We let the limit phase function be
\be \label{II}{\II}(x,t)  =  \begin{cases}  - 1  & \hbox{ if  } (x,t) \in \OO^-\\   +1 &\hbox{ if  }  (x,t)\in \OO^+ \end{cases} . \ee

Our main theorem is the following.

\begin{theorem} \label{teo1}
For each $j\in \mathbb N$,
there exist initial conditions $u_0^\ve$, $u_1^\ve$ for a solution $u_\ve(x,t)$ of problem $\equ{1}$ with the property that
$$
u_\ve(x,t) \to \II(x,t) \ass \ve \to 0 \qquad\mbox{ in the $C^{j+1}$ sense}
$$
in compact subsets of  $([0,T] \times \R^n)\setminus \NN$.
Inside $\NN$ we have
$$u_\ve (x,t) =  w\left (\frac z\ve \right ) + \phi_\ve (x, t), \quad (x,t) =  Y +  z\nu(Y).  
$$
and
\be |\phi_\ve (x,t) |+ |D_x^{j+1} \phi_\ve (x,t) |+ |D_x^j \pp_t \phi_\ve (x,t) | \ \le \    C\ve. 
\label{phive.est}\ee
\end{theorem}

The proof of Theorem \ref{teo1} involves various ingredients with a simple philosophy:
 First we obtain an expansion in powers of $\ve$ of a true solution that gives an arbitrarily algebraic high order of approximation in $\ve$. After this approximation is built, estimates for the remainder with a sufficiently good control are found. This is a delicate step in which positivity of the one-variable quadratic form associated to the linearization of the equation \equ{homo} is essential, as well as designing well-adapted systems of coordinates.

The proof provides much more precise information about the solution. In fact, for a given number $k\ge 1$  we can find a solution that near $\Gamma$
takes the form
\be\label{forma}
u_\ve (x,t) =  w\left (\frac z\ve -h^*_\ve (Y) \right ) + \phi^*_\ve (x,t) +  \vp_\ve (x,t),  \quad (x,t) =  Y +  z\nu(Y),
\ee
where $h_\ve^*, \phi^*_\ve$ are explicit functions with
$h_\ve^* = O(\ve)$, $\phi^*_\ve = O(\ve^2)$ in smooth sense, and the remainder $\vp_\ve$ satisfies
  \be |\vp_\ve (x,t) |+ |D_x^{j+1} \vp_\ve (x,t) |+ |D_x^j \pp_t \vp_\ve (x,t) | \ \le \    C\ve^k. 
\label{phive.est1}\ee
The solution described is stable in the sense that smooth perturbations of its initial condition with size $O(\ve^m)$ and sufficiently large $m$ produce a solution with the same qualitative features. This rules out exponential growth of small perturbations (which in general may happen).

\subsection{More about prior work}\label{S:comp}
As noted above, our proof of Theorem \ref{teo1}
constructs solutions $u_\ve$ whose behavior we are
able to describe  to arbitrary precision, in arbitrarily strong norms.
The best (indeed, the only) prior results construct
solutions $u_\ve$ that satisfy the weaker estimate
\[
\| u_\ve - \II\|_{L^2((0,T')\times \R^N)}\le C(T')\ve^{1/2} \qquad\mbox{ for any } T'<T,
\]
together with some weighted estimates that quantify energy concentration
around $\Gamma$. This was proved in  \cite{J1}
under the assumption that $\Gamma_0$ is a topological torus, but allowing rather general initial velocity for $\Gamma$ --- more general in fact than we consider here. The proof in \cite{galjer} assumes that $\Gamma_0$ has zero initial velocity but allows it to be an arbitrary smooth connected compact manifold, among a number of generalizations of the results in \cite{J1}.

It has been recently proved in \cite{es-j} that when $n=2$,
one can extract from the weighted energy estimates in  \cite{J1, galjer}
an estimate of the form
\[
\| u_\ve - u_\ve^*\|_{L^2((0,T')\times \R^2)}
+ \ve \| D(u_\ve - u_\ve^*)\|_{L^2((0,T')\times \R^2)}
\le C(T')\ve^{3/2} \mbox{ for any }T'<T
\]
for some $u_\ve^*$ whose description is less explicit than the one that we construct
in this paper.
This seems to be the limits of the precision
attainable by the strategy employed in prior work,
and it also seems only to be available in $2$ space
dimensions.

Another drawback of the technique of \cite{J1, galjer,es-j} is that these results
rely on standard well-posedness theory to provide solutions of \eqref{1}. This imposes
growth conditions that render these papers unable to address the canonical cubic
nonlinearity $f(u) = u(1-u^2)$ in
high dimensions. No such growth conditions are needed in this paper.

Related prior results on issues that  we do not address
include the following:
\begin{itemize}
\item In \cite{galjer},  equations like \eqref{1}, but with asymmetric nonlinearities for which there is a bias toward one of the potential wells, are shown to have solutions with an interface that approximately sweeps out a timelike hypersurface of constant (nonzero) Minkowskian mean curvature.
\item In \cite{J1}, a Ginzburg-Landau wave equation --- like \eqref{1}, but for a complex-valued function $u$, with nonlinearity $f(u) = u(1-|u|^2)$ ---  is shown to have solutions for which
energy concentrates near a codimension $2$ surface of vanishing Minkowskian mean curvature.
\item Results of a similar character are proved for the Abelian Higgs model in \cite{Cz-J},
a Ginzburg-Landau wave equation coupled to a wave equation for an electromagnetic potential,
for certain values of a coupling constant appearing in the equations.
\item A scattering result is proved in \cite{sc}  for  \eqref{1} in $\R^{1+3}$ for initial data
$(u, u_t)|_{t=0}$  a small, very smooth perturbation of  $(w(x^3), 0)$. This can be seen as an analog for \eqref{1} of results \cite{lind, Brendle}
that establish scattering  Minkowski minimal surfaces
with initial data that is a small perturbation of a motionless hyperplane.
\end{itemize}
We believe that it should be possible to strengthen at least some  of the above results by the methods that we introduce here.

\section{Construction of an approximation}

\subsection{The wave operator in Fermi coordinates}

Let us consider a general smooth, orientable  $n$-dimensional manifold $\Gamma$ embedded in $\R^{n+1}$
and let $\NN$ be a small tubular neighborhood of $\Gamma$ defined by relation \equ{fer}.
We will find an expression for the wave operator acting on functions
$u(x,t)$ with    $(x,t)\in \NN$,
$$
\Box u =  - \pp^2_t u + \Delta_x u   \inn \NN
$$
when $u$ is expressed in {\em Minkowskian Fermi coordinates} that
we introduce next. All points in $\NN$
can be uniquely represented in the form
$$
(x,t)= Y + z\nu(Y), \quad Y \in \Gamma, \quad |z| < \delta .
$$
provided that $\delta$ is taken sufficiently small.

\medskip
Let assume that $\Gamma$ is compact and parametrized by a finite number of smooth maps 
$$ y\in \Lambda_i \subset \R^n   \ \mapsto\  Y_i (y)\in \R^{n+1}, \quad i\in I, $$
so that
$$
\Gamma = \bigcup_{i\in I} Y_i(\Lambda_i).
$$
We also use the convention $y=(y_0,\ldots, y_{n-1}) $.

Define
$$
I[u] =\iint_{\R^{n+1}}  ( |\nn_x u(x,t) |^2   - |u_t(x,t)|^2  )\, dx \, dt
$$
where $u(x,t)$ is a smooth function supported sufficiently close to a compact portion of the manifold $\Gamma$. We write
$$  \nn  u(x,t)    =   \left [ \begin {matrix}\  \pp_t  u (x,t)\\   \nn_x  u (x,t)\end{matrix}   \right ]  .$$
Then
$$
I[u] =\iint_{\R^{n+1}}    \nabla  u(x,t)^T J  \nabla  u(x,t) \, dx \, dt
$$

Let us assume for the moment that $u$ is supported close to one of the coordinate patches $Y_i(\Lambda_i)$.
Let us omit the subindex $i$ in the pair $(\Lambda_i, Y_i)$ and consider  local coordinates in a neighborhood of  $Y(\Lambda)\subset \Gamma$ given by
$$
(x,t) :=  Y(y) + z\nu(y),  \quad y\in \Lambda \subset \R^n,\quad |z|<\delta
$$
where we are just  setting $\nu(y) := \nu (Y(y))$. We refer to $(y,z)$ as {\em Fermi coordinates} associated to the local coordinate system $Y:\Lambda\to \Gamma$.
Let us write
$$
v(y,z) =   u(x,t),  \quad (x,t)= Y(y) + z\nu(y)  .
$$
and
$$  \nn  v(y,z)    =   \left [ \begin {matrix}  \nn_y v(y,z) \\ \ \pp_z v(y,z)  \\   \end{matrix}   \right ]  .
$$
We use the following notation
$$
Y_a =  \pp_{y_a} Y + z\pp_{y_a} \nu  ,
 \quad =0,1,\ldots, n-1; \qquad Y_n = \nu
 $$
and
\[
g_{\a\b}(y,z) =  \langle Y_\a, Y_\b \rangle_m , \qquad \a, \b = 0,\ldots, n.
\]
We will call $g(y,z)$ the matrix of entries $[g(y,z)]_{\a\b} = g_{\a\b}(y,z)$.
We will also denote
\[
g^{\a\b}(y,z) = [g(y,z)^{-1}]_{\a\b}
\]
and
$$
g^0_{ab} (y) =   g_{ab}(y,0), \qquad a,b =0\ldots, n-1.
$$
Consistent with this, we will always tacitly assume that $\a,\b,\ldots$  run from $0$ to $n$,
and $a,b,\ldots$ run from $0$ to $n-1$, and we will sum over repeated indices.

We introduce the matrix
$$
B = \big [Y_0\cdots  Y_n \big ],   
$$
and we remark that $B^T J B = g$.
The definition of the Minkowskian normal $\nu = Y_n$
directly imply the following basic property of Fermi coordinates:
\begin{equation}
\begin{aligned}\label{basic.fermi}
g_{an} = g_{na} = \metric{Y_a,\nu} &=0\qquad\mbox{ for }a=0,\ldots, n-1,\\
g_{nn} = \metric{\nu,\nu} &= 1.
\end{aligned}
\end{equation}
From Chain's rule we find
$$
  \nn  u(x,t)  \   = \   B^{-T} \, \nn v (y,z) , \quad   (x,t) = Y(y) + z\nu(y).
$$
and hence
$$
(\nn u)^T J \nn u\ = \     (\nn v )^T  A^{-1} \, \nn v
$$
where
\be\label{op}
A = B^T J B = g 
\ee
and hence (using \eqref{basic.fermi})
$$
A^{-1} = g^{-1} =  (g^{\a\b})_{\a,\b=0}^n = \left [ \begin{matrix} (g^{ab})_{a,b=0}^{n-1} \   &\ 0  \\  0 \ & \ 1  \end{matrix} \right ].
$$
A related observation is that
\be\label{gg}
\sqrt{|\det g (y,z)|} \ =\  \left | \det \big [ B \, \big|\, \nu\,  \big ]\right |
\ee
Indeed, we have
 $$  \big | \det \big [B \, \big|\, \nu\,  \big ] \, \big |  \ =\ \sqrt{ |\det A|  }  $$
where
$A$ is the matrix in \equ{op}
and \equ{gg} readily follows.

Then we find that
\begin{align*}
I[u]
&=
 \iint_{\R^{n+1}}  \nabla v^T    g(y,z) ^{-1} \nabla v   \sqrt {\big |\det g(z) \big | } \, dy\, dz \\
 &= \  \iint_{\R^{n+1}} g^{\a\b}(y,z)\, \pp_\a v\, \pp_\b v \, \sqrt {\big |\det g(z) \big | } \, dy\, dz .
\end{align*}
Taking a test function $\vp(x,t)$ supported within the range of validity of these local coordinates we find

$$
- \frac 12 \frac d{d\lambda} I[u + \la \varphi  ]  \big |_{\lambda = 0 } =  \int_{\R^{n+1}} \Box u(x,t)\,  \varphi (x,t)\, dt\, dx
$$
hence letting $\psi (y,z) = \varphi (Y(y) + z\nu(y)) $ we find
$$
\begin{aligned}
\iint_{\R^{n+1}} \Box u(x,t)\,  \varphi (x,t)\, dt\, dx \ = &\  \iint_{\R^{n+1}}  \mathcal L [v](y,z) \,  \psi(y,z) \,\sqrt {|\det g (y,z) | } \, dy\, dz \\ = &\  \iint_{\R^{n+1}}  \mathcal L [v] \, \vp  \, dt\, dx
\end{aligned}
$$
where
\begin{equation}\label{Box.general}
 \mathcal L [v] =   \frac 1{ \sqrt {|\det g(y,z) | }}\partial_{\a} [    \sqrt {|\det g(y,z) | } g^{\a\b}(y,z) \pp_{\b} v ].
\end{equation}
Recalling the form of $(g^{\a\b})$, this simplifies to
$$
 \mathcal L [v] =   \frac 1{ \sqrt {|\det g(y,z) | }}\partial_{a} [    \sqrt {|\det g(y,z) | } g^{ab}(y,z) \pp_{b} v ]
 + \frac 1{ \sqrt {|\det g(y,z) | }}\partial_{z } ( \sqrt {|\det g(y,z)| } \partial_z v )
 $$
 and $g^{ab}(y,z) = [g(y,z)^{-1}]_{ab}$.
The following definition is in order. For a sufficiently small $z$, the wave operator associated to a time-like manifold
$$
\Gamma_z =  \{ Y+ z\nu(Y) \ /\ Y\in \Gamma\}
$$
is given by
$$
\Box_{\Gamma_z}\ = \  \frac 1{ \sqrt {|\det g(y,z) | }}\, \partial_{a} [    \sqrt {|\det g(y,z) | } g^{ab}(y,z)\,  \pp_{b} \,  ]
$$
which acts on functions of the local coordinate  $y$ for $\Gamma_z$.
The {\em mean curvature} in the Minkowski sense of the manifold $\Gamma_z$  at the point  $Y(y) + z\nu(y)$
is defined by the quantity
$$
H_{\Gamma_z}(y) \ =\   -\frac 12 \frac{\pp}{\pp z} \log |\det g(y,z)|, \quad
$$
so that correspondingly the Minkowskian mean curvature of $\Gamma$ at the point  $Y= Y(y)$ is given by
\be
H_{\Gamma}(Y) \ =\   -\frac 12 \frac{\pp}{\pp z} \log |\det g(y,z)|\, \big|_{z=0} \quad
\label{mean}\ee
For a function $f$ defined on $\Gamma$ we will write indistinctly $f(Y)$ or $f(y)$ when $Y= Y(y)$, with reference to local coordinates.

In summary, we have proven the validity of the formula
\be\label{exp}
\Box  =  \Box_{\Gamma_z}   + \partial^2_z    -  H_{\Gamma_z} \partial_z\ .
\ee

\medskip
At this point we establish the key definition.
\begin{definition} \label{minimal}
A time-like hypersurface $\Gamma$ in $\R^{n+1}$ is said to be {\bf minimal}  in the Minkowski sense if  its Minkowski mean curvature given by \equ{mean} vanishes:
\be
H_{\Gamma}(Y) \ =\  0\foral Y\in \Gamma.
\label{mean1}\ee

\end{definition}

In what follows we will always assume that $\Gamma$ is minimal. We can then write
\be\label{H}
H_{\Gamma_z}(y) = 
 z\, {\bf a}_\Gamma (y)  + z^2{\bf b} _\Gamma (y,z)
\ee
The justification of the notation $\Box_{\Gamma_z}$ comes from the fact that the matrix
$
g (y,z)
$
defining the metric, and hence the operator in local coordinates has all positive eigenvalues except one which is negative.
This is a consequence of the time-like character of the surface $\Gamma$. To see this, we check that in the case of a non-vertical time-like plane in $\R^{n+1}$. We can parametrize it in the form
$$
(x,t) =  ( \alpha\cdot  y , y ), \quad y=(y_1,\ldots, y_n).
$$
where $\alpha =(\alpha_1,\ldots, \alpha_n)$.
A normal vector to this plane is $(-1,\alpha)$ and the time-like character clearly corresponds to the relation
$|\alpha|^2 > 1$. In this case we directly compute
$$
(B^T J B)_{ij}    = \delta_{ij} - \alpha_i\alpha_j.
$$
We see that $1$  is an eigenvalue of this matrix with multiplicity $n-1$, while its trace is negative because of the time-like condition. Hence in addition, exactly one negative eigenvalue is present.  It follows from this fact and the compactness of $\Gamma$ that
(after shrinking $\delta$ if necessary)
\be\label{detg.negative}
\det g(y,z) \le - c < 0 \quad\mbox{ everywhere in }\Gamma\times (-\delta, \delta).
\ee
We will introduce in \S \ref{coords} local coordinates under which $\Box_\Gamma$  truly becomes a wave operator.

\subsection{Shift of coordinates and construction}\label{ss:approximation}
Our purpose is to find a good approximate solution for  the equation
\be\label{eqq}
S(u) :=  \ve^2 \Box u+ f(u) = 0\quad
\ee
valid in a small $\ve$-independent neighborhood of the manifold $\Gamma$, that has a sharp transition layer near $\Gamma$.
More precisely, let us consider a heteroclinic  $w(\zeta)$ as defined in \eqref{homo}.
Taking into account expression \equ{exp} for the wave operator and \equ{H}, we see that equation \equ{eqq}
can be written as
$$
S(u)= \ve^2 \pp^2_z u +  f(u)  +  \ve^2 \Box_{\Gamma_z} u- \ve^2( {\bf a}_{\Gamma}  + z {\bf b}_\Gamma ) z\pp_z u = 0  ,\quad
$$
where we write  $$ u(y,z)= u(x,t)   \quad \hbox{ for } (x,t) = Y(y) + z\nu(y). $$
We take as a first approximation, in the small neighborhood of $\Gamma$, $|z|< \delta$,  $u_0 (x,t) = w\left ( \frac z\ve \right ) $.
In that region we get
$$
S(u_0)= -  \ve^2( {\bf a}_{\Gamma}(y)  + z {\bf b}_\Gamma(y,z) ) \zeta \pp_{\zeta} w(\zeta)\,|_{\zeta =\ve^{-1}z}  = O(\ve^2 e^{-\frac {a|z|}\ve}) .\quad
$$
To be observed is that the fact that $\Gamma$ is a Minkowskian minimal surface yields that the order of approximation on the interface is
$\ve$ times better. As we will see, more than this: is will be possible to slightly modify $u_0$ so that the order of approximation is
$ O(\ve^k e^{-\frac {|z|}\ve})$ for any given $k\ge 2$.
Indeed, as we will see one can find  a function $u^k(y,z)$ that achieves this property in the region $|z|<\delta$ for a small $\delta$  with the form
\be\label{ansatz1}
 u^k(x,t) =  v^k(y,\zeta), \qquad  \zeta = \frac z \ve  - h^k(y), \qquad (x,t) = Y(y) + z\nu(y)
\ee
where
\be\label{ansatz2}
v^k(y, \zeta) =  w(\zeta ) + \phi^k(y,\zeta), \quad \phi^k(y,\zeta) = O(\ve^2 e^{-|\zeta|} ) .
\ee
For a given, sufficiently small function $h$ defined on $\Gamma$ and a function $v(y,\zeta)$ of  the form
$$  u(x,t) =  v( y,  \ve^{-1}z-  h(y) )  , \quad (x,t) =  Y(y) +  z\nu(y)   .  $$
we compute
$$
\begin{aligned}
S(u)\  = & \ \ve^2 \Box u(x,t)+ f(u(x,t)) \, =\,  S(v,h)    \end{aligned}
$$
where
$$
\begin{aligned}
 S(v,h) \ := & \  \pp^2_\zeta v(y,\zeta) +  f( v(y,\zeta)) \,+\,  \LLL_\ve(z , h) [v](y,\zeta) \, \big|_{z= \ve(\zeta +  h(y)) },\\
\LLL_\ve(z , h) [v]  \  =  &\
  \ve^2 \Box_{\Gamma_z} v   -   \ve^2 \Box_{\Gamma_z} h\, \pp_\zeta v  -  \ve z ({\bf a}_{\Gamma} + z{\bf b}_{\Gamma})  \pp_\zeta v \\
& + \ve^2\langle \nn_{\Gamma_z}h , \nn_{\Gamma_z}h \rangle \pp^2_\zeta v    -
2\ve^2\langle \nn_{\Gamma_z}\pp_\zeta v , \nn_{\Gamma_z}h \rangle \,
\end{aligned}
$$
and we have denoted, for functions $h_1(y)$, $h_2(y)$,
$$\langle \nn_{\Gamma_z} h_1(y), \nn_{\Gamma_z} h_2(y)\rangle \,\  =\  g^{ab}(y,z) \pp_a h_1(y) \pp_b h_2(y)  .  $$

\medskip
To construct a first approximation,
we let $v^0(y, \zeta)  = w(\zeta) +\phi^0(y,\zeta)$.
Choosing $h=0$ we get
\be\label{poto3}
\begin{aligned}
S(v^0,0)\  & =  \ \pp_\zeta^2\phi^0   +  f'(w)\phi^0   - \ve^2 {\bf a}_{\Gamma} (y)\zeta w'(\zeta)  - \, \ve^3 {\bf b}_\Gamma(y,\ve \zeta)\zeta^2  w'(\zeta)\\
&\qquad
 +\, \ve^2\Box_{\Gamma_{\ve\zeta}}\phi^0
 -  \ve^2 \left({\bf a}_{\Gamma} (y)\zeta   - \, \ve {\bf b}_\Gamma(y,\ve \zeta)\zeta^2 \right)\pp_\zeta \phi^0
  + N(\phi^0)   ,
\end{aligned}
\ee
\[
\]
where
$$
N(\phi)\  =  \  f(w(\zeta)  + \phi )- f(w(\zeta) ) -  f'(w(\zeta) )\phi. $$
A basic property that we will use is that the equation
$$
p''(\zeta)  +  f'(w(\zeta))p(\zeta) + q(\zeta)   =0 , \quad \zeta\in [-R,R]
$$
has the solution
\be\label{formulita}
p(\zeta) =    \mathcal T  [ q(\zeta)] :=  w'(\zeta) \int_{-R}^\zeta  w'(s)^{-2}  \left( \int_{-R}^s q(\tau) w'(\tau) d\tau \right) ds.
\ee
We have that if
\be\label{or}
\int_{-R}^R q(\tau) w'(\tau) d\tau  =0 ,
\ee
and for $j\ge 0$
$$
\quad   |D^j q(\zeta)|\ \le  \ (1+|\zeta|^m) e^{-a|\zeta|} ,\quad \zeta\in [-R,R],
$$
then
\be\label{estim}
|D^j p(\zeta)|\ \le  \ C_j\, (1+|\zeta|^{m+1}) e^{-a|\zeta|},\quad \zeta\in [-R,R].
\ee
with $C$ uniform in all large $R$.
At this point we observe that since $w$ the  heteroclinic is odd by assumption,
$q(\zeta) = \zeta w'(\zeta) $ satisfies \equ{or} for any $R>0$.
We now let
$$
\phi^0 (y,\zeta ) = -\ve^2 {\bf a}_{\Gamma} (y) \mathcal T [\zeta w'(\zeta)].
$$
Using \equ{estim} we see that for $j,l\ge 0$ we have
$$
| D_y^lD_\zeta^j \phi^0 (y,\zeta)|\ \le  \ C_{jl}\ve^2 \, (1+|\zeta|) e^{-a|\zeta|},\quad  |\zeta|\le \frac \delta\ve .
$$
Moreover, the first three terms in expansion \equ{poto3} are identically cancelled and the resulting error
gets one order smaller. Indeed, we directly check that
$$
 | D_y^lD_\zeta^j S(v^0,0)(y,\zeta) | \  \le \ C_{jl}   \ve^3 \, (1+|\zeta|^2)e^{-a|\zeta|},\quad  |\zeta|\le \frac \delta\ve \ .
$$
This procedure can be continued inductively but involves adjusting the function $h(y)$ to get the orthogonality conditions
\equ{or} satisfied. As we will see, such an adjustment will involve an equation for $h$ that involves the Jacobi-Minkowski operator
of the minimal surface $\Gamma$. More precisely we need to solve  equations on $\Gamma$ of the form

\be \label{J1}
\begin{aligned}
J_\Gamma [h] :=   \Box_\Gamma h  +  {\bf a}_\Gamma(y)  h \, =& \, g \inn \Gamma \\
 h \,= \, \pp_t h \,= &\, 0 \onn \Gamma\cap\{t=0\}.
\end{aligned}
\ee

\begin{lemma}\label{lema1}  Let $g$ a function of class $C^\infty (\Gamma)$. Then Problem \equ{J1}
has a unique solution $h$ which is also of class $C^\infty (\Gamma)$.
Moreover for each $j\ge 0$ there are numbers $m_j$, $C_j$ such that
$$
\| D_y^j h\|_{L^\infty(\Gamma)} \le C_j \sum_{l=0}^{m_j} \| D_y^l g\|_{L^\infty(\Gamma)}.
$$
\end{lemma}
The proof consists of reducing the problem to one for a standard wave-like operator.
We postpone it for the appendix.
Our main result in this section is the following.

\begin{prop}\label{prop.as} Given $k\ge 0$
there exists smooth functions $h^k(y)$ and $\phi^k(y,\zeta)$, with $\phi^k$  defined in the set
$$
\DD= \{ (y,\zeta) \ :\ y\in \Gamma,\  - \frac \delta {2\ve} < \zeta < \frac \delta {2\ve} \},
$$
such that for all $j,l\ge 0$
\be\label{pico1}
\begin{aligned}
| D_y^lD_\zeta^j \phi^k (y,\zeta)|\ \le &  \ C_{jlk}\ve^2 \, (1+|\zeta|) e^{-a|\zeta|},\quad
| D_y^lD_\zeta^j h^k (y,\zeta)|\ \le & \ C_{jk}\ve
\end{aligned}
\ee
and for $v^k(y,\zeta) = w(\zeta) + \phi^k(y,\zeta)$
 we have
\be\label{pico2}
| D_y^lD_\zeta^j  S( v^k , h^k) |  \le     C_{ljk}\, \ve^{k+3} (1+ |\zeta|^{k+2} ) e^{-a|\zeta|}\inn \DD.
\ee
\end{prop}

\begin{proof}
We proceed by induction. The case $k=0$ has just been dealt with with the choice $h^0= 0$.
Let us assume the existence of functions $h^k$, $\phi^k$ as in \equ{pico1}-\equ{pico2}. We will make a choice for   $h^{k+1}$, $\phi^{k+1}$.

\medskip
Let us consider two functions $h(y)$ and $\phi(y,\zeta )$ with the following properties: for a certain $m$
 and each numbers $j,l$, there are constants $C_{ljk}$, $C_{jk}$ such that for all sufficiently small $\ve$ we have
\begin{align}
|D^l_\zeta D_y^j \phi(y,\zeta)| \ \le & \   C_{ljk}\ve^{k+3}(1+  |\zeta|^{k+2}) e^{-a|\zeta|}, \label{coco1}\\
   | D_y^j h(y)| \ \le &\  C_{jk} \ve^{k+1}. \label{coco2}
\end{align}
We explicitly find functions that satisfy constraints of this form such that $h^{k+1} = h^k+h$ and
$\phi^{k+1}= \phi^k+\phi$  reduce the error, thus completing the induction step.

\medskip
We expand in the region  $\DD$,
\be
\begin{aligned}
 \LLL_\ve(\ve(\zeta +h^k + h ) , h^k+ h) [v^k+ \phi]  \  =  &\  \LLL_\ve(\ve(\zeta +h^k ) , h^k) [v^k] \\ &
 +  \big ( \LLL_\ve(\ve(\zeta +h^k + h ) , h^k+ h)  - \LLL_\ve(\ve(\zeta +h^k  ) , h^k)  \big)  \, [v_k] \\ & +  \LLL_\ve(\ve(\zeta +h^k + h ) , h^k+ h) [\phi]  \\ = &\quad \LLL_\ve(\ve(\zeta +h^k ) , h^k) [v^k]  +\,  \ve^2 [\Box_{\Gamma} h\,  +    {\bf a}_\Gamma h  ]  \pp_\zeta w \\
   &\ +\,   \Theta_1 (h,\phi)
\end{aligned}
\nonumber\ee
where the remainder $\Theta_1 (h,\phi)$ satisfies
\be
|D^l_\zeta D_y^j \Theta_1 (h,\phi)|\ \le \     C_{ljk}\ve^{k+4}(1+ |\zeta|^{k+3})e^{-|\zeta|}
\label{poto}\ee
for some constants relabeled $C_{lj}$. Hence we find
\be\label{poto1}
\begin{aligned}
&S_\ve(v^k+ \phi , h^k+ h)
\\\ = \ &        \pp^2_\zeta \phi  +  f'(w(\zeta))\phi  +   S_\ve(v^k, h^k) +  \ve^2 [\Box_{\Gamma} h\,  +    {\bf a}_\Gamma h  ]  \pp_\zeta w  +  \Theta (h,\phi)
\end{aligned}
\ee
where $\Theta$ satisfies an estimate of the form \equ{poto}. Next we choose the function $h$:
We consider $h(y)$ such that the following relation holds.
\be\label{or1}
\int_{-\frac \delta{2\ve}}^{\frac \delta{2\ve}} \EE(y,\zeta) \, \pp_\zeta w(\zeta)\, d\zeta\ =\ 0 \foral y\in \Gamma.
\ee
 where
$$
\EE(y,\zeta) =    S_\ve(v^k, h^k)(y,\zeta) +  \ve^2 [\Box_{\Gamma} h(y)\,  +   {\bf a}(y) h(y)  ]  \pp_\zeta w(\zeta),
$$
We can write this equation in the form
$$
J_\Gamma [h](y)= \Box_{\Gamma} h(y)\,  +   {\bf a}_{\Gamma}(y) h(y)  =  g(y) \onn \Gamma
$$
where the function  $g(y)$ satisfies that for each $j\ge 0$
$$
|D^j_y g(y)| \ \le\   C_{jk}\ve^{k+1}\inn \Gamma
$$
Assuming the initial conditions $h = \pp_t h = 0 $ on $\Gamma\cap \{t=0\} $  we see from Lemma \ref{lema1} that a unique solution $h$
of this problem exists which also satisfies a bound of the form \equ{coco2}.
Now, we choose $\phi(y,\zeta)$ to be the solution of the equation

$$
\pp_\zeta^2  \phi  +  f'(w(\tau))\phi + \EE(y,\zeta) = 0 , \quad |\zeta|< \frac\delta\ve
$$
given by
$$
\phi(y,\zeta) = \mathcal T [ \EE(y,\zeta)]
$$
with $\mathcal T$ as in \equ{formulita}. Using
Estimate \equ{estim} we get that
$\phi$ satisfies the bounds
$$
| D_y^lD_\zeta^j \phi (y,\zeta)|\ \le  \ C_{jlk}\ve^{k+ 3}\, (1+|\zeta|^{k+2}) e^{-|\zeta|},\quad  |\zeta|\le \frac \delta\ve .
$$
With these choices of $h$ and $\phi$ made, we indeed have the validity of \equ{coco1}. Hence
setting $v^{k+1}= v^k +\phi, \quad h^{k+1}= h^k+ h$ we get
$$
S(v^{k+1},h^{k+1})\ =\  \Theta (h,\phi)
$$
which satisfies bounds \equ{poto}. The induction is thus complete and the proposition follows.
\end{proof}

\subsection{The global approximation} \label{ss:global}

We have built in Proposition \ref{prop.as} an approximation to a solution of $S(u)$ =0 of the form
$$  u_\ve^k (x,t) =  w(\ve^{-1}z-  h_k(y) ) +   \phi^k( y,  \ve^{-1}z-  h_k(y) )  , \quad (x,t) =  Y(y) +  z\nu(y)   ,  $$
which is only defined in the small neighborhood $\NN$ of $\Gamma$.  We can obtain a globally defined approximation  by just interpolating with the function $\II$ defined in \equ{II} as follows.
Let us consider a smooth, nonnegative cut-off function $\eta(s)$ such that $\eta(s)= 1$ for $s<1$ and $=0$ for $s>2$, and set
\be\label{chi00}
\chi_0 (x,t)  =  \eta\left (\frac {|z|}{r}\right ) ,
\ee
where $2r< \delta$ and this function is understood as zero whenever $(x,t)$ is outside the neighborhood of $\Gamma$ of points with coordinate $|z| <\delta$ and $r$ is a sufficiently small number which we will specify at the beginning of Section \ref{sec.linear},
Then we define
\be
u^*_\ve (x,t) :=   \chi_0(x,t) u_\ve^k (x,t)  + (1- \chi_0(x,t)) {\II}(x,t)
\label{uestar}\ee
where the number $k$ will be chosen sufficiently large.

\section{Further coordinate systems}

\subsection{A canonical coordinate system in $\Gamma$ }\label{coords}

Let us consider a time-like manifold $\Gamma$ endowed with local parametrizations
$$(Y_l,\Lambda_l), \quad l=1,\ldots, m .$$
The tangent space to $\Gamma$  at the point $Y= Y_l(y)$ is the $n$-dimensional space
$$
T_Y\Gamma\, =\, {\rm Span}\, \{ \pp_iY_l (y) \ /\ i=0,\ldots, n-1\}.
$$
We denote $\Gamma^t$ the $t$-section of $\Gamma$, namely
$$
\Gamma^t = \Gamma \cap\{ (x,t) \ / x\in \R^n\}
$$
We claim that if $\Gamma^t$ is nonempty, it is a $n-1$ dimensional smooth manifold.
Indeed, let
 of $\Gamma$.
Writing
$$
Y_l(y)  = (t_l(y), x_l(y)), \quad y\in \Lambda_l  \subset \R^n.
$$
Then $\Gamma^t$ is locally parametrized by the equations $t_l(y) = t$. This set is a smooth manifold. In fact,
$ \nn_y t_l (y) \ne 0 $
In fact if  $ \nn_y t (y) =  0$ we would have that $T_Y\Gamma$ at $Y= Y_l(y)$  is just $\{0\}\times \R^n$.  Hence an Euclidean
 normal vector is $e_0 =(1,0,\ldots, 0) $. That contradicts the time-like condition.
The section $\Gamma_t $ has an $n-1$-dimensional tangent space $T_Y\Gamma_t$ contained in  $\{0\}\times \R^n$.
We consider a vector $E(Y)\in T_Y\Gamma$ which lies in the orthogonal to  $T_Y\Gamma^t$. We make the unique choice of this vector with
$$
E(Y)\cdot e_0 = 1 .
$$
The map $Y\in \Gamma \mapsto E(Y) \in T_Y\Gamma$ defines a smooth vector field on $\Gamma$ which we will use to define
a natural system of local coordinates that will be helpful for computations.

\medskip
Natural coordinates on $\Gamma$ are those associated to {\em flow lines} for the vector field $E$. These are the trajectories of the differential equation on the manifold $\Gamma$
\be
\frac {d Y} {ds }(s)  = E( Y(s)) , \quad Y(s)\in \Gamma.
\label{ss}\ee
The meaning of this equation is given by local coordinates as
$Y(s) =  Y_l (y(s)) $, where $y(s)\in \Lambda_l\subset \R^n$ solves the system of equations
$$
DY_l (y(s)) [\frac {dy} {ds}]     =      E( Y_l(  y(s)) )
$$
or equivalently the system of ODEs
$$
\frac {dy} {ds}(s)   =  F(y(s))  $$
where $$
F(y) =  [(DY_l(y) )^T(DY_l(y))]^{-1} (DY_l(y))^T E( Y_l(y)).
$$
For each point $Y^0= (0,x_0)\in \Gamma^0$,  Equation \equ{ss} has a unique solution $Y(s)$ with $Y(0)= Y^0$ which we denote as
$Y(s ,x_0) $. To be observed is that by definition of $E$, this function has the form
$$ Y(t ,x_0) = (t, X(t;x_0)) $$
where $X(0;x_0)= x_0$.

 \medskip
 Using this map we can define local coordinates on $\Gamma$ just based on
coordinates on $\Gamma^0$
We regard
$\Gamma^0 $ as a $n-1$ dimensional manifold in $\{0\}\times \R^{n}$.
 We consider a family of smooth maps
 $ X^0_l : V_l \subset \R^{n-1} \to  \R^n$
 with the functions
$$
y' = (y_1,\ldots , y_{n-1})\in V_l  \mapsto    (0,X^0_l (y')) \in \Gamma^0
$$
 defining local coordinates for $\Gamma^0$. Then the following maps define local coordinates that parametrize entire $\Gamma$.
 We let $T>0$ be any number such that $\Gamma^s $ is nonempty for all $0\le s\le T$
and define
$$ \Lambda_l =  [0,T]\times V_l, \quad l=1,\ldots, m
$$
 and consider the maps $Y_l$ defined as
 \be\label{canonical.coords}
 Y_l(y_0,y')  =     Y( y_0,   X^0_l(y') )   = (y_0,  X(  y_0,   X^0_l(y')))
 \ee

\medskip
Let us consider the Minkowski metric $g^0(y) $ associated to this parametrization, defined on $\Gamma$ as
\be\label{metric}
g_{ab}^0(y) \ := \  \langle\pp_a Y_l(y)  ,   \pp_b Y_l(y)\rangle , \quad a,b =0,\ldots n-1.
\ee

\begin{lemma}\label{lema2}
The following properties of the metric $g^0$ defined above hold:
\begin{align}
\label{11}
g_{0a}^0\  =&  \  0, \quad a=1,\ldots n-1,\\
g^0_{00}\  <& \ 0 . \quad  
\label{condi}
\end{align}
The matrix $\bar g^0$ with coefficients
$$  [\bar g^0]_{ij} =  g^0_{ij}\quad   i,j=1,\ldots n-1 $$
is positive definite.
\end{lemma}

\begin{proof}

To be noticed is that for $i=1,\ldots, n-1$ and $Y=  Y_l(y)$ we have that $\pp_a Y_l \in T_Y\Gamma^t$ and
$
\pp_{0} Y_l = E( Y) .
$

Since $J\pp_a Y_l= \pp_a Y_l$, then by definition of $E$ we have that
$$
g_{0a}  = \langle\pp_0 Y_l  ,   \pp_i Y_l\rangle  = 0, \quad i=1,\ldots n-1.
$$

Next, let us observe that for $i,j=1,\ldots, n-1$ we have that
$$
\langle \pp_i Y_l , \pp_j Y_l   \rangle =  \pp_i X_l \cdot \pp_j X_l
$$
Besides the $n-1$  vectors $ \xi_i(t) :=  \pp_i X_l(t,y')$ are linearly independent.
This follows from the fact that all vectors  $\xi_i (t)$ are solutions of a linear system of the form
$$
\frac {d \xi_i}{dt}  (t) = A(t)[\xi_i (t)].
$$
They are linearly independent at $t=0$ since they are associated to local coordinates for the manifold $\Gamma^0$, and that property is preserved in time. The matrix $\Xi(t)$ whose columns are $\xi_i(t)$ is therefore non-singular, hence the matrix
$$\bar g^0(y_0,y') = \Xi(y_0)^T\Xi(y_0)$$ is positive definite.

Finally, \eqref{11} and the positive definiteness of $\bar g^0$ implies that
$g_{00}$ and $\det g$ have the same sign, so \eqref{condi} follows from
\eqref{detg.negative}.

 The proof is concluded. \end{proof}

\bigskip
A nice characteristic of the local coordinates built above is that they allow to
express the
$\Box_\Gamma$ operator in a clean way as a second order wave operator.  That leads to a clean proof of Lemma \ref{lema1}

\subsection{Proof of Lemma \ref{lema1} }
We want to solve the equation
\be \label{J2}
\begin{aligned}
  \Box_\Gamma h  +  {\bf a}_\Gamma  h \, =& \, g \inn \Gamma \\
 h \,= \, \pp_t h \,= &\, 0 \onn \Gamma\cap\{t=0\}.
\end{aligned}
\ee
for a given function $g$. In local coordinates around
$\Lambda_l = [0,T] \times V_l $ the equation is expressed as

\be\label{cort}
\frac 1{\sqrt{|\det g ^0(y)|}}    \pp_a( \sqrt{|\det g ^0(y)|}  g^{0,ab}(y) \pp_b h)   +  {\bf a}_\Gamma(y))  h \, = \, q(y) \
\ee
As customary we write $g^{0,ab}(y) $ for the entries of the matrix $(g^0(y)) ^{-1}$.
From the previous lemma, we see that
$$
(g^0(y)) ^{-1}=  \left [ \begin{matrix} (g^0_{00}(y) )^{-1} \   &\ 0  \\  0 \ & \ (\bar g^0(y))^{-1}   \end{matrix} \right ].
$$
and hence, naturally relabeling  $y= (t,y')$, \equ{cort} can be written in the coordinate patch $[0,T]\times V_l$ in the form
\be\label{onda}
\begin{aligned}
-  \pp^2_th  +    a_{ij}(t,y')\pp_{ij} h  +b_0(t,y')\pp_t h + b_i(t,y')\pp_i h   + \bar {\bf a}_\Gamma (t,y') h   =  Q(t,y'), &\\
 (t,y')\in [0,T]\times  V_l, & \\
h(0,y') =   h_t (0,y')= 0,  \quad y'\in V_l.&
\end{aligned}
\ee
for certain coefficients $b_\alpha$, where
$$
\begin{aligned}
a_{ij} (t,y')\,= &  \,   g_{00}(t,y') \,  g^{0,ij}(t,y'), \\   \bar {\bf a}_\Gamma(t,y')\, = &\,g_{00}(t,y')\
 {\bf a}_\Gamma(t,y'), \\   Q(t,y')\, =&\,g_{00} (t,y') \, q (t,y').
\end{aligned}
 $$
The matrix with entries $a_{ij} (t,y)$ is uniformly positive definite.
If we consider a smooth bounded domain  $\bar \Omega \subset V_l$ and restrict equation \equ{onda} to $\Omega$ with zero boundary conditions,
the standard theory for linear wave equations based on energy estimates, as developed in \cite{evans}, Section 7.2 yields existence and regularity with uniform controls
in Sobolev spaces of arbitrary order in terms of corresponding norms of $Q$.
Existence of a solution of the full problem \equ{J2} follows from a standard argument using a partition of unity on $\Gamma$, while uniqueness is a byproduct of energy identities. That solution clearly has uniform controls as stated thanks to Sobolev embeddings. The proof is complete.

\subsection{modified Fermi coordinates}\label{ss.mfc}

In our later arguments, we will need to glue together estimates close to and far from
$\Gamma$. In order to do this, it is convenient to introduce a new coordinate system
in a neighborhood of $\Gamma$.  These will coincide with
Fermi coordinates near $\Gamma$ and farther from $\Gamma$, they
will have the property that the timelike variable coincides with the  $t$ variable of standard $(x,t)$ coordinates.

We will mostly\footnote{except in the proof of
Lemma \ref{L.modferm} in Appendix \ref{AppA}, in which we
need to distinguish carefully between
the different coordinate systems.}
abuse notation somewhat and not distinguish between Fermi coordinates
and modified Fermi coordinates. Thus we will continue to
write the coordinates as $(y,z)$, where $y = (y_0,y')
= (y_0,\ldots, y_{n-1})$. Similarly, we will generally write $g_{\a\b}$
to denote the metric tensor with respect to these coordinates.

To state the main properties of this coordinate system, we first need to
introduce some notation.
Let
\be\label{clp}
Y_l : \Lambda_l\to \Gamma, \quad l=1\ldots, m \qquad\mbox { for }
\Lambda_l = [0,T]\times V_l
\ee
be the canonical local parametrizations fixed in Section \ref{coords}.
With this notation, given $T_1<T$,
the modified Fermi coordinate system will be defined locally via
\[
(x,t) = \Phi_l(y, z), \qquad (y,z)\in [0,T_1]\times V_l\times(-\delta_1,\delta_1)
\]
for some map $\Phi_l: [0,T_1]\times V_l\times (-\delta_1,\delta_1)\to \R^{1+n}$ and some
$\delta_1\le \delta$, both
 constructed in  the proof of Lemma \ref{L.modferm} below.
We will choose these maps to be independent of $l$
in the sense that
\be\label{ind.of.l}
\mbox{ if $Y_l(y) = Y_k(\widetilde y)$ for $y\in \Lambda_l$ and $\widetilde y\in \Lambda_k$, then $\Phi_l(y,z) = \Phi_k(\widetilde y, z)$.}
\ee
This implies that $\{ \Phi_l\}_l$ will induce a well-defined function
\[
\Phi: [0,T_1]\times \Gamma^0\times (-\delta_1,\delta_1)\to \R^{1+n}
\]
defined by setting $\Phi(y_0, X^0_l(y'), z) := \Phi_l(y_0, y', z)$.
We will abuse notation somewhat and write $\Phi$ to mean either this
function or else its representative $\Phi_l$ with respect to a
generic local parametrization $X^0_l:V_l\to\Gamma^0$, depending on the context.

We will use the notation $ \left[ \begin{array}{r}  g_{\a\b}
\end{array}
\right]_{\alpha, \beta = 0}^n$  for
components of the metric tensor in local
coordinates:
\[
g_{\alpha\beta}(y,z):=   \metric{\frac{\pp \Phi_l} {\pp y_\alpha}, \frac {\pp\Phi_l}{\pp y_\b}}, \qquad \qquad \alpha,\beta = 0,\ldots, n, \mbox{ with }\frac{\pp}{\pp y_n} :=\frac{\pp}{\pp z} .
\]
Similarly, $ \left[ \begin{array}{r}  g^{\a\b}
\end{array}
\right]_{\alpha, \beta = 0}^n$ denotes the inverse metric tensor.

\begin{lemma}\label{L.modferm}
There exists coordinates as described above
and numbers $r_2<r_1 \le \delta_1/2$
with the following
properties.

First, $\Phi(y,z) = Y(y)+z\nu(y)$ for $|z| < r_2$,
where $Y$ is the canonical  parametrization of $\Gamma$ from \eqref{canonical.coords},
and thus
\begin{equation}\label{ggl.gab}
\left. \begin{aligned}
\gf_{0i} = \gf^{0i} &=O(|z|) \ \mbox{ for }i=1,\ldots, n\\
\gf_{an}= \gf^{an} &=0 \ \mbox{ for }i=0,\ldots, n-1\\
\gf_{nn} = \gf^{nn}&=1\\
\end{aligned}
\right\} \qquad\mbox{ when $|z|<r_2$}
\end{equation}
and
\begin{equation}
\left. \frac {\pp}{\pp z}\, \det (g(y,z)) \right|_{z=0} = 0 \ . 
\label{dz.gnn}\end{equation}
Second,
\begin{equation}\label{y0=tbis}
y_0 = t \quad\mbox{ when }(x,t) = \Phi(y,z), \mbox{ \ \  if \ \  }
\begin{cases}|z|\ge r_1 &\mbox{ or }
\\y_0 = 0.&
\end{cases}
\end{equation}
Finally,
\begin{equation}\label{gab.signs}
\left. \begin{aligned}
\gf_{00}, \ \gf^{00}&<0\\
 \left[ \begin{array}{r}  \gf_{ij}
\end{array}
\right]_{i,j=1}^{n},
 \left[ \begin{array}{r}  \gf^{ij}
\end{array}
\right]_{i,j=1}^{n} &\mbox{ are positive definite}
\end{aligned}
\right\}
\end{equation}
everywhere in $[0,T_1]\times V_l\times(-\delta_1,\delta_1)$ for all $l$.

\end{lemma}

We defer the proof to Appendix \ref{AppA}.
Conclusions \eqref{ggl.gab} and \eqref{dz.gnn}
will be immediate from our construction and from properties of Fermi coordinates noted above.
Properties  \eqref{y0=tbis}
will be useful when we patch together energy estimates near and far from  $\Gamma$. Condition \eqref{gab.signs} is the point in the proof that requires the most attention. It is needed to
guarantee coercivity of energy estimates computed with respect to this coordinate
system.

\medskip

\section{Linear theory}\label{sec.linear}

We are interested in linear estimates associated to the operator
\be\label{Lve}
L_\ve [\vp] :=  \Box \vp +  \frac 1 {\ve^2}f'(u_\ve^*) \vp
\ee
obtained
by linearizing \eqref{1} around the global approximate solution $u_\ve^*$
constructed in
Sections \ref{ss:approximation} and \ref{ss:global}. We first introduce some notation.

Recall that the construction of modified Fermi coordinates introduced
induces a map $\Phi:[0,T_1]\times \Gamma^0\times(-\delta_1,\delta_1)\to \R^{1+n}$.
For
$s\in [0,T_1]$ we will write
\[
\Sigma_s^{nr}  := \{ \Phi(s, y', z) : y'\in \Gamma^0, |z|<\delta_1 \}, \\
\]
Our standing assumptions imply that $\Gamma_s = \{ \Phi(s, y', 0) : y'\in \Gamma^0\}$
divides $\{s\}\times \R^n$ into two disjoint open components, say $\OO^+_s$ and $\OO^-_s$, with $\OO^-_s$ being bounded.
The same thus holds for
$\Gamma_{s,z} := \{ \Phi(s, y', z) : y'\in \Gamma^0\}$ whenever $r_1\le |z| <\delta_1$,
since then \eqref{y0=tbis} imples that  $\Gamma_{s,z}$ is a subset of $\{s\}\times \R^n$
that retracts onto $\Gamma_s$.
For $s \in [0,T_1]$ we define
\begin{align*}
\Sigma^-_s
&:= \mbox{ the bounded component of }(\{ s\}\times \R^n)\setminus \Gamma_{s, -r_1}\\
\Sigma^+_s
&:= \mbox{ the unbounded component of }(\{ s\}\times \R^n)\setminus \Gamma_{s, r_1}\\
\Sigma^{far}_s
&:= \ \Sigma^+_s \cup \Sigma^-_s\\
\Sigma_s
&:= \ \Sigma^{nr}_s \cup \Sigma^{far}_s\\
\Sigma
&:= \cup_{s\in [0,T_1]} \Sigma_s\, .
\end{align*}
For $0<\rho \le \delta_1$ we will also  use the notation
\begin{equation}
\NN_\rho  = \{ \Phi(y, y',z):  (y_0,y', z)\in   [0,T_1]\times\Gamma^0 \times (-\rho, \rho) \} .
\label{Nrho.def}\end{equation}

Next, we specify that the cutoff function $\chi_0$ in the definition of $u_\ve^*$
satisfies
\be\label{chi0}
\chi_0=1 \mbox{ in } \NN_{r_2/4}, \qquad
\chi_0 = 0 \mbox{ in }\Sigma\setminus \NN_{r_2/2}.
\ee
The exponential decay of the local approximate solution $u_\ve$ away from $\Gamma$
implies that for every $l$, $\sup_{\NN_{\delta_1}}|D^l (u_\ve - u_\ve^*)| \le C e^{-c/\ve}$ for suitable constants
$C,c$ (depending on $l$). Hence in $\NN_{\delta_1}$, writing $u_\ve^*$ as a function of modified Fermi coordinates
$(y,z)$,
\be
u_\ve^*(y,z) = w_\ve(y,z) +\phi(y,z), \qquad\quad w_\ve(y,z) = w\left (\frac z\ve - h(y)\right )
\label{ustar.recall}\ee
where $w$ is the heteroclinic \eqref{homo} and $h,\phi$ satisfy \eqref{pico1}.
We will prove

\begin{prop}\label{prop.linear}
Given a smooth function $\eta\in L^2({\Sigma})$  and smooth
data $(\vp_0,\vp_1) \in H^1\times L^2(\R^n)$,
there exists a smooth
solution $\vp: \Sigma \to \R$
to the initial value
problem
\begin{equation}
L_\ve [\vp] = \eta \ \ \mbox{ in } \Sigma,
\qquad\qquad
(\vp, \pp_t\vp)\Big|_{t=0} = (\vp_0,\vp_1)  .
\label{linear.eqn}\end{equation}
In addition, there exist $C, \ve_0>0$, depending only on $\Gamma$ and $T_1$ and $\delta_1$,
such that
and for every $s\in [0,T_1]$ and $\ve\in (0,\ve_0)$, we have the estimate
\be
\begin{aligned}\label{crude.linear}
&\int_{\Sigma_s} \ve^2\left(|\nabla_x\vp|^2 +(\partial_t\vp)^2]\right) + \vp^2
\\
&\qquad\qquad\qquad \le
C \int_0^s\Big( \int_{\Sigma_\sigma} \eta^2 \Big) d\sigma +
C\int_{\R^n}\left[ |\nabla_x \vp_0|^2 +  |\vp_1|^2+ \frac 1{\ve^2}\vp_0^2 \right] dx .
\end{aligned}
\ee
\end{prop}

In \eqref{crude.linear},
the integrals over $\Sigma_s$ are with respect to the induced Euclidean $n$-dimensional volume. (In fact we will employ a variety of different $n$-forms in our arguments, but all of them are uniformly comparable to the Euclidean $n$-volume.)

The  point of the proposition is the estimate; existence of a solution is standard
and we will not discuss it.

Note that even when $\eta = 0$, the estimate as stated allows the terms on the left-hand side to be larger by a factor of $\ve^{-2}$ than the corresponding terms on the right-hand side.
In fact our proof yields a sharper estimate, see Remark \ref{linear.sharper}.
However, \eqref{crude.linear} is sufficient for our later purposes.

\medskip

\begin{proof}[Proof of Proposition \ref{prop.linear}]

{\bf 1}.
Our overall aim is to construct some quantity $E(s)$ that controls
the left-hand side
of \eqref{crude.linear}, and that satisfies a differential
inequality allowing for the application of Gr\"onwall's inequality.
This quantity will be constructed by integrating an energy density
over $\Sigma_s$ with respect to well-chosen $n$-forms. We will
treat $\Sigma^{nr}_s$ and $\Sigma^{far}_s$ separately.

We start by describing the $n$-form we will employ on $\Sigma^{nr}_s$.
First, fix a volume form on
$\Gamma^0$, which we will write in local coordinates as
$\omegazero(y')dy'$.
Let  $\chi^{nr}:\R\to [0,1]$ be a smooth function
such that
\[
-C(r_1) \le \partial_z \chi^{nr} \le 0,
\qquad \chi^{nr}(z)=1 \mbox{ for }z\le  r_1,
\qquad \chi^{nr}(z)=0 \mbox{ for }z\ge 2r_1
\]
for $r_1$ defined in \eqref{y0=tbis}.
Then we define $\omega^{nr}_s$ to be the $n$-form on $\Sigma^{nr}_s$ written in
local coordinates as
\[
\omega^{nr}_s = \omegazero(y')\chi^{nr}(z) dy'\, dz .
\]
Thus
for any $s\in [0,T_1]$ and  function $f = f(s,y',z)$ on $\Sigma_s^{nr}\cong \{s\}\times \Gamma^0\times (-\delta_1,\delta_1)$,
\[
\int_{\Sigma^{nr}_s} f \omega^{nr}_s
=
\int_{-\delta_1}^{\delta_1} \int_{y'\in \Gamma^0} f(s,y',z) \, \omegazero(y') dy' \, \chi^{nr}(z)\,dz.
\]

We next define $\omega^{far}_s$ as the $n$-form on $\Sigma^{far}_s$ written in $(x,t)$ coordinates on
as
\[
\omega^{far}_s := \chi^{far}(s,x) \, dx
\]
where $\chi^{far}\in C^\infty(\OO)$ satisfies
\[
\chi^{far} = 1\mbox{ outside of }\NN_{\delta_1}
\]
and in $\NN_{\delta_1}$, writing $\chi^{far}$ as  a function of modified Fermi coordinates $(y,z)$,
\[
\chi^{far}(y,z) = 1 - \chi^{nr}(z).
\]
Finally we will define $\omega_s := \omega^{far}_s+\omega^{nr}_s$, an $n$-form on $\Sigma_s$
that is uniformly comparable to the induced Euclidean $n$-volume.

{\bf 2}.
We now derive an energy identity in modified Fermi coordinates near $\Gamma$.
In doing so, we will write the solution $\vp$  of \eqref{linear.eqn} as a function of modified Fermi coordinates $(y,z)$,
and we will identify $\frac{\pp}{\pp y_n}$ with $\frac \pp {\pp z}$.
In these coordinates,
we find\footnote{In the discussion that contains \eqref{Box.general}, we were interested in Fermi
coordinates, but \eqref{Box.general} is completely general, and the particular choice of coordinates was used only later. In any case this is standard.} from \eqref{Box.general} that
\eqref{linear.eqn}  has the form
\[
\frac {1}{\sqrt{|\det g | }}\frac \pp {\pp y_\a}\left(\sqrt{|\det g|} g^{\a\b}\frac {\pp \vp} {\pp y_\b }
\right)
+  \frac 1 {\ve^2} f'(u_\ve^*)\vp = \eta .
\]
We rewrite this as
\[
\frac 1{\omegazero} \frac \pp {\pp y_\a}\left(\omegazero g^{\a\b} \frac{\pp \vp}{\pp y_{\b}} \right)
+b^\b  \frac{\pp \vp}{\pp y_{\b}}
+  \frac 1{\ve^{2}}f'(u_\ve^*)\vp = \eta,
\]
for
\[
b^\b :=\frac \omegazero { \sqrt {|\det g|}}g^{\a\b}\frac \pp {\pp y_\a} \left(\frac{ \sqrt{|\det g|}}\omegazero  \right).
\]
We multiply this equation by $ \pp_{y_0}\vp$ and rewrite.
This gives rise to a number of terms.
An easy term is
$$
\ve^{-2} f'(u_\ve^*) \vp \frac{\pp \vp}{\pp y_0} = \ve^{-2}\frac {\pp }{\pp y_0}\left (  \frac 12  f'(u_\ve^*) \vp^2 \right ) -
\ve^{-2}  \frac {\vp^2}2 \frac {\pp }{\pp y_0}  ( f'(u_\ve^*)) ,
$$
The error term
$ { b}^\beta \pp_\beta \vp \, \pp_0\vp$ we keep as it is.
The leading term is rewritten as follows:
\begin{align*}
\frac 1 \omegazero
\frac \pp {\pp y_\a}\left( \omegazero g^{\a\b}  \frac {\pp \vp }{\pp y_\b}
\right)\frac {\pp\vp }{\pp y_0}
&=
\frac 1 \omegazero\frac \pp {\pp y_\a}\left(\omegazero  g^{\a\b}
\frac {\pp \vp}{\pp y_\b}  \frac {\pp\vp }{\pp y_0}
\right)
-  g^{\a\b} \frac {\pp \vp}{\pp y_\b}  \ \frac {\pp^2\vp} {\pp y_\a \pp y_0}
\\
&=
\frac 1 \omegazero
\frac \pp {\pp y_\a}\left(\omegazero  g^{\a\b} \frac {\pp \vp}{\pp y_\b}
\frac {\pp\vp }{\pp y_0}
\right)
- \frac1 2 \frac {\pp }{\pp y_0}(  g^{\a\b}
\frac {\pp \vp}{\pp y_\b}  \frac {\pp\vp }{\pp y_\a}
)
\\&
\qquad\qquad\qquad
+ \frac1 2 \frac{ \pp g^{\a\b}} {\pp y_0}
\frac {\pp \vp}{\pp y_\b}  \frac {\pp\vp }{\pp y_0}.
\end{align*}

We substitute these computations into the equation,
write  $ \frac \pp {\pp y_\a} (\cdots )$ as $ \frac \pp {\pp y_0} (\cdots ) +  \frac \pp {\pp y_i} (\cdots )$,
where $i$ runs from $1$ to $n$,
and rearrange to obtain
\begin{align}
& \frac {\pp }{\pp y_0} \left( -g^{0\b}
 \frac {\pp \vp}{\pp y_\b}  \frac {\pp\vp }{\pp y_0}
 + \frac 12 g^{\a\b}
 \frac {\pp \vp}{\pp y_\a}  \frac {\pp\vp }{\pp y_\b}  -   \frac 1{2\ve^2} f'(v) \vp^2 \right)
\label{E0}
\\&\qquad\qquad =
\frac 1 \omegazero \frac {\pp }{\pp y_i}
 \left(    \omegazero g^{i\b}
\frac {\pp \vp}{\pp y_\b}  \frac {\pp\vp }{\pp y_0}
\right )
+ \frac 1 2 (\frac{\pp}{\pp y_0}g^{\a\b})
\frac {\pp \vp}{\pp y_\a}  \frac {\pp\vp }{\pp y_\b}
\nonumber
\\&\qquad \qquad\qquad\qquad
-
 \frac {\vp^2} {2\ve^2}
\frac{\pp}{\pp y_0} (f'(u_\ve^*))
+
 {b}^\b  \frac {\pp \vp}{\pp y_\b}  \frac {\pp\vp }{\pp y_0}
-
\eta  \frac {\pp\vp }{\pp y_0} .
\nonumber
\end{align}

We introduce the tensor $(a^{\a\b})$, defined by
\begin{equation}
a^{00} = - g^{00}, \qquad\qquad
a^{0i} = a^{i0} = 0
\qquad\qquad
a^{ij} =  g^{ij}
\label{aab0}\end{equation}
for $i,j=1,\ldots, n$. It is then easy to check
that $(a^{\a\b})$ is positive definite, and that
\begin{equation}
-g^{\a\b} \xi_0 \xi_\b + \frac 12 g^{\a\b}\xi_\a \xi_\b =
\frac 12 a^{\a\b}\xi_\a \xi_\b
\qquad \qquad\mbox{ for all }\xi.
\label{aab}\end{equation}
Note also that the quadratic form
$a^{ab} \pp_a \vp\,\pp_b \vp$ does not
depend
on the choice of local coordinates on $\Gamma^0$.
We will write
\begin{equation}\label{enr.def}
e_\ve^{nr}(\vp) =
 \left (  \frac 12 a^{\alpha\beta}
\frac {\pp \vp}{\pp y_\a}  \frac {\pp\vp }{\pp y_\b}
 - f'(u_\ve^*) \frac{\vp^2}{2\ve^2}  \right ).
\end{equation}
With this notation, \eqref{E0} becomes
\be
\label{E0bis}
\begin{aligned}
 \frac {\pp }{\pp y_0}\Big(  \ e_\ve^{nr}(\vp) \Big)
&= \frac 1 \omegazero \frac {\pp }{\pp y_i}
 \left(    \omegazero g^{i\b}
\frac {\pp \vp}{\pp y_\b}  \frac {\pp\vp }{\pp y_0}
\right )
+ \frac 1 2 (\frac{\pp}{\pp y_0}g^{\a\b})
\frac {\pp \vp}{\pp y_\a}  \frac {\pp\vp }{\pp y_\b}
\\&\qquad \qquad\qquad
-
 \frac {\vp^2} {2\ve^2}
\frac{\pp}{\pp y_0} (f'(u_\ve^*))
+
 {b}^\b  \frac {\pp \vp}{\pp y_\b}  \frac {\pp\vp }{\pp y_0}
-
\eta  \frac {\pp\vp }{\pp y_0} .
\end{aligned}
\ee

We will also write
\[
E^{nr}_\ve(s;\vp) =
E^{nr}_\ve(s) =
\int_{\Sigma^{nr}_s}e_\ve^{nr}(\vp) \ \omega^{nr}_s  =
\int_{\Sigma^{nr}_s}e_\ve^{nr}(\vp)  \omegazero(y') \chi^{nr}(z) dy'\, dz
\]

{\bf 3}. We will next integrate \eqref{E0bis} with respect
to the $n$-form $\omega^{nr}_s$ over $\Sigma^{nr}_s$.
First note that since $\omegazero$ and $\chi^{nr}$ are independent of $y_0$,
\begin{align*}
\int_{\Sigma^{nr}_s}
 \frac {\pp }{\pp y_0}\Big(  \ e_\ve^{nr}(\vp) \Big)\omega^{nr}_s
&=
\int_{-\delta_1}^{\delta_1}\int_{\Gamma^0} \frac {\pp }{\pp y_0}\Big(  \ e_\ve^{nr}(\vp) \Big)
\chi^{nr}(z)\omegazero(y') \, dy' \, dz\Big|_{y_0=s}
\\
&=
\frac d{ds}E^{nr}_\ve(s) .
\end{align*}

Next, for every $z\in (-\delta_1,\delta_1)$, let $X(\cdot; z)$ denote the vector
field on $\Gamma^0$ whose $i$th component in local coordinates on $\Gamma^0$ is given by $X^i(y') = g^{i\beta}\frac{\pp\vp}{\pp y_\beta}\frac{\pp \vp}{\pp y_0}(s, y', z)$.
For every fixed $z$, the divergence of $X(\cdot, z)$ on $\Gamma^0$ with respect to the $n-1$- form
$\omegazero(y')dy'$ is
\[
\operatorname{div}_{\omegazero} X = \frac 1 \omegazero \sum_{i=1}^{n-1} \frac{\pp}{\pp y_i}( \omegazero X^i)
\]
Since $\Gamma_0$ is a compact manifold without boundary,
$\int_{\Gamma_0}( \operatorname{div}_{\omegazero}X)\  \omegazero(y')dy' = 0$.
Thus
\begin{align*}
\int_{\Sigma^{nr}_s}
\frac 1 \omegazero \frac {\pp }{\pp y_i}
\left(    \omegazero g^{i\b}
\frac {\pp \vp}{\pp y_\b}  \frac {\pp\vp }{\pp y_0}
\right ) \omega^{nr}_s
&
= \int_{-\delta_1}^{\delta_1}
\int_{\Gamma^0}
\frac {\pp }{\pp z}
\left(    \omegazero g^{n\b}
\frac {\pp \vp}{\pp y_\b}  \frac {\pp\vp }{\pp y_0}
\right )dy' \, \chi^{nr}(z)\, dz \Big|_{y_0=s}\\
&= - \int_{-\delta_1}^{\delta_1}
\int_{\Gamma^0}
g^{n\b}
\frac {\pp \vp}{\pp y_\b}  \frac {\pp\vp }{\pp y_0}
(\chi^{nr})'(z) \omegazero(y') dy' \, \, dz \Big|_{y_0=s}.
\end{align*}
By integrating \eqref{E0bis} we thus obtain
\be \label{Enear1}
\begin{aligned}
\frac {d }{d s} E^{nr}_\ve(s)
&
=
- \int_{\delta_1}^{\delta_1}\int_{\Gamma^0}
 g^{n\b}
\frac {\pp \vp}{\pp y_\b}  \frac {\pp\vp }{\pp y_0}
(\chi^{nr})'(z)\omegazero(y') dy' dz \Big|_{y_0=s}
 \\ &\qquad\qquad\qquad+
 \int_{\Sigma^{nr}_s}
 \left( -\frac {\pp }{\pp y_0}(  f'(u_\ve^*)) \frac{\vp^2} {2\ve^2}  +
b^\b
 \frac {\pp \vp}{\pp y_\b}  \frac {\pp\vp }{\pp y_0}
 \right) \omega^{nr}_s
\\&\qquad \qquad\qquad
+
\int_{\Sigma^{nr}_s} \left(
  \frac 1 2 (\pp_0g^{\a\b})
\frac {\pp \vp}{\pp y_\a}  \frac {\pp\vp }{\pp y_\b}
-\, \eta \, \frac {\pp\vp }{\pp y_0}
\right) \omeganr.
\end{aligned}
\ee

{\bf 4}.
We next derive a (completely standard) parallel identity far from $\Gamma$.

The counterpart of $(a^{\a\b})$, defined as in \eqref{aab0}, but starting from the Minkowski
metric tensor $J$  in standard $(x,t)$ coordinates, is
just the identity tensor $(\delta^{ab})$. We thus define
\[
e^{far}_\ve(\vp) :=
\frac 12\left[ (\pp_t\vp)^2 + |\nabla _x\vp|^2- f'(u_\ve^*)\frac{\vp^2}{\ve^2}\right]
=
\frac 12\left[ (\pp_t\vp)^2 + |\nabla _x\vp|^2+ \frac \sigma {\ve^2}\vp^2 \right]
\]
where $\sigma = -f'(\pm 1)$ and
we have used the fact that
$u_\ve^* = \pm1$ in $\Sigma^{far}$. We also set
\[
E^{far}_\ve(s;\vp) :=
E^{far}_\ve(s) =
\int_{\Sigma^{far}_s}e_\ve^{far}(\vp) \ \omega_s^{far} .
\]
Then arguments like those in the derivation of \eqref{Enear1},  but significantly
easier, lead to the identity
\begin{align}
\frac {d}{ds} E^{far}_\ve(s)
 \label{Efar1}
&
= -\int_{\Sigma^{far}_s}
 \frac {\pp\vp }{\pp t}
\frac {\pp \vp}{\pp x_i} \frac {\pp \chi^{far}}{\pp x_i}  dx
+ \int_{\Sigma^{far}_s} e^{far}_\ve(\vp) \frac{\partial \chi^{far}}{\partial t} \, dx
 \\&\qquad \qquad\qquad
+
\int_{\Sigma^{far}_s} -\, \eta \, \frac {\pp\vp }{\pp y_0}
\chi^{far}(s,x)\, dx.\nonumber
\end{align}

{\bf 5}.
As mentioned earlier, we plan to construct a quantity $E(s)$ which
satisfies a differential inequality allowing for the application of Gr\"onwall's inequality.
This will have the form
\begin{equation}\label{Es.def}
E(s) := E^{nr}_\ve(s) + E^{far}_\ve(s) + \frac C\ve \int_{\Gamma^0} \gamma(s,y')^2
 \omegazero(y')dy
\end{equation}
where $\gamma:\Gamma\to \R$ is a function defined in \eqref{gamma.def}
below, arising as a component in a decomposition of $\vp$, and $C$ is also fixed below. These are needed to guarantee that
$E(s)$ bounds suitable norms of $\vp$; this is not completely straightforward,
since the quantity $-\ve^{-2}f'(u_\ve^*) \vp^2$ appearing in $E^{nr}_\ve(s)$
is negative in places.

We next
derive the relevant bounds. In doing so, we will define
$\gamma$ and fix the constant $C$ in \eqref{Es.def}.

It is convenient to decompose $E^{nr}_\ve(s)$ into pieces.
Recall from Lemma \ref{L.modferm} that  modified
Fermi coordinates coincide with actual
Fermi coordinates in $\NN_{r_2}$. To take advantage of this, we fix
$\chi_1: \R\to [0,1]$ be a smooth function such that
\[
|\partial_z \chi_1| \le C(r_2),
\qquad \chi_1(z)=1 \mbox{ for }|z|\le \frac 12 r_2\, ,
\qquad \chi_1(z)=0 \mbox{ for }|z|\ge r_2\, .
\]
We then split  $E^{nr}(s)$ into two pieces as follows.
\[
I_1 := \int_{\Sigma^{nr}_s} e^{nr}_\ve(\vp) \chi_1^2(z) \, \omega^{nr}_s,
\qquad
I_2 := \int_{\Sigma^{nr}_s} e^{nr}_\ve(\vp) (1-\chi_1^2(z)) \, \omega^{nr}_s.
\]
Concerning $I_2$, we only note that we have arranged in \eqref{uestar}, \eqref{chi0}
that $u_\ve^* = \pm 1$, and hence
$-f'(u_\ve^*) = \sigma$ when $|z|\ge r_2/2$,
so
\begin{equation}\label{I2.est}
I_2  =  \frac 12 \int_{\Sigma^{nr}_s}\left( a^{\a\b}\pp_\a\vp \, \pp_\b \vp + \frac \sigma{\ve^2}\vp^2\right)
(1 - \chi_1^2(x)) \omega^{nr}_s.
\end{equation}
Next, it follows from \eqref{aab0}
and properties of Fermi coordinates, see \eqref{ggl.gab},  that
\[
I_1 =
\frac 12 \int_{\Gamma^0}\int_{-r_2}^{r_2} \left ( a^{ab} \pp_a \vp\pp_b \vp
+ (\pp_z\vp)^2 - \frac 1  {\ve^2} f'(u_\ve^*) {\vp^2} \right )\chi_1^2(z) \,
 \omegazero(y')\, dz\,  dy'.
\]
(For the duration of the estimate of $I_1$,  all integrals are evaluated at $y_0= s$.)
We define $\bar \vp (y',z) := \vp(y',z) \chi_1(z)$. Then
\begin{align}
I_1
&=
\frac 12 \int_{\Gamma^0}\int_{-r_2}^{r_2}  a^{ab} \pp_a  \bar \vp\ \pp_b \bar \vp\
 \omegazero(y')\,  dz\, dy'  \nonumber\\
&\qquad  +
\frac 12 \int_{\Gamma^0}\int_{-r_2}^{r_2} \left (
 (\pp_z\bar \vp)^2 - \frac 1  {\ve^2} f'(u_\ve^*) {\bar \vp^2} \right )
 \omegazero(y')\,  dz\, dy'  \nonumber\\
&\qquad  -
\int_{\Gamma^0}\int_{-r_2}^{r_2}\left( \frac 12 (\chi_1')^2 \vp^2 +
\chi_1 \, \chi_1' \, \vp\, \pp_z \vp  \right)\,\omegazero(y')\, dz \, dy'
\nonumber \\
&=
I_{1,1} + I_{1,2} - I_{1,3}. \label{I1split}
\end{align}
It is clear that $I_{1,1}$ is positive definite.
We write $\vp  \pp_z\vp = \frac 12\pp_z(\vp^2)$ and integrate by parts to obtain
\begin{equation}\label{i1s}
|I_{1,3}|
 \le C \int_{\Gamma^0}\int_{-r_2}^{r_2}
\vp^2   \omegazero(y')\, \chi^{nr}(z)\,dz\, dy' \ .
\end{equation}
Since
$\vp ^2
= \chi_1^2 \vp^2 + (1 - \chi_1^2) \vp^2 = \bar \vp^2 + (1 - \chi_1^2) \vp^2$,
it follows from \eqref{I2.est} that
\[
|I_{1,3}| \le C\left( \ve^2 I_2 +  \int_{\Gamma^0}\int_{-r_2}^{r_2}
\bar \vp^2  \omegazero(y')\, dz\, dy' \right) .
\]

\medskip

We now split $\bar \vp$ as
\begin{equation}\label{bphi.split}
\bar \vp(y,z) := \bar \vp^\perp(y,z) + \gamma(y)\pp_z w_\ve(y,z),
\end{equation}
where  $w_\ve (y,z) = w(\frac z\ve - h(y))$ and
\begin{equation}\label{gamma.def}
\gamma(y) := \frac \ve{  \Xi }\int_\R \bar \vp(y,z) \pp_z w_\ve(y,z) dz,\qquad
\qquad\Xi := \int_{\R} w'^2(\zeta)d\zeta.
\ee
The definition implies that
\be\label{ortho}
\int_\R \bar \vp^\perp(y,z) \pp_z w_\ve(z) dz = 0
\end{equation}
for all $y$ and hence that
\be
\int_{\Sigma^{nr}_s} \bar \vp^2 \omega^{nr}_s =
\int_{\Gamma^0} \int_\R( \bar \vp^\perp)^2 \omegazero(y')\,dz\, dy'\  +\frac \Xi \ve \int_{\Gamma^0}\gamma^2 \omegazero(y')dy'.
\label{split.L2}\ee
In terms of $\bar \vp^\perp$ and $\gamma$, our above estimate of $I_{1,3}$
takes the form
\begin{align}
|I_{1,3}|
&\le C\ve^2 I_2 +
C\int_{\Gamma^0}\int_{\R}
(\bar \vp^\perp)^2   dz\, \omegazero(y')\,dy'
+
\frac C \ve \int_{\Gamma^0}\gamma^2 \omegazero(y')dy' \ .
\label{I12.est}\end{align}

Turning to $I_{1,2}$, and omitting ``$dz$" and ``$\omegazero(y')dy'$" when no confusion can result, we now set $f'(u_\ve^*) = f'(\pm 1) = \sigma$ for $|z|\ge r_2$, and we rewrite
\begin{align*}
I_{1,2}
&=
\frac 12 \int_{\Gamma^0}\int_\R
 (\pp_z \bar \vp^\perp)^2
- \frac 1 {\ve^2} f' (u_\ve^*) (\bar \vp^\perp )^2 \\
&\qquad\qquad +  \int_{\Gamma^0} \int_\R
 \gamma\  \pp_{z}^2 w_\ve\,  \partial_z(\bar\vp^\perp + \frac \gamma 2 \pp_z w_\ve)
-  \frac 1 {\ve^2}f' (u_\ve^*) \gamma \pp_z w_\ve (\bar\vp^\perp + \frac \gamma 2 \pp_z w_\ve).
\\
&=
I_{1,2,1} + I_{1,2,2} .
\end{align*}
We integrate by parts in the $z$ variable and use the fact that $\pp_{z}^3w_\ve  + \ve^{-2}f'(w_\ve)\pp_z w_\ve=0$
to find that
\[
I_{1,2,2}
=
\int_{\Gamma^0} \int_\R \frac 1 {\ve^2}(f'(w_\ve) -  f'(u_\ve^*)) \, \gamma\, \pp_z w_\ve \,
 (\bar\vp^\perp + \frac \gamma 2 \pp_z w_\ve).
\]
It follows from \eqref{ustar.recall}, \eqref{pico1} that $|f'(w_\ve) - f'(u_\ve^*)| \le C\ve^2$
everywhere, and as a  result,
\begin{equation}
|I_{1,2,2}|\le C \int_{\Gamma^0} \int_\R (|\gamma \pp_z w_\ve| \, |\bar\vp^\perp| + \gamma^2(\pp_z w_\ve)^2 )
\le \frac C \ve \int_{\Gamma^0} \gamma^2
+  C\int_{\Gamma^0} \int_\R (\bar \vp^\perp)^2 .
\label{I112.est}\end{equation}


Next, because $\int \bar\vp^\perp \pp_z w_\ve = 0$,
it follows from \equ{quadratic}  that there exists some $c>0$ such that
\[
\int_\R
(\pp_z \bar \vp^\perp)^2
- \frac 1 {\ve^2} f' (w_\ve) (\bar \vp^\perp )^2\, dz \, \ge
c \int_{\R} (\pp_z \bar \vp^\perp)^2 + \frac 1 {\ve^2}(\bar\vp^\perp)^2 dz
\]
for every $y'$.
Arguing as with \eqref{I112.est}, we infer that
\begin{align}
I_{1,2,1}
&\ge
c \int_{\Gamma^0}\int_{\R} (\pp_z \bar \vp^\perp)^2 + \frac 1 {\ve^2}(\bar\vp^\perp)^2
\, \omegazero(y')\,dz\,dy' \nonumber \\
&\qquad\qquad +
 \frac 12 \int_{\Gamma^0}\int_\R
 \frac 1 {\ve^2} (f'(w_\ve) - f' (u_\ve^*)) (\bar \vp^\perp )^2
 \, \omegazero(y')\,dz\,dy' \nonumber \\
 &\ge
c \int_{\Gamma^0}\int_{\R} (\pp_z \bar \vp^\perp)^2 + \frac 1 {\ve^2}(\bar\vp^\perp)^2
\, \omegazero(y')\,dz\,dy'  \label{I121}
\end{align}
for $\ve$ sufficiently small.

Next, we consider $I_{1,1}$. By compactness, there exists some $c>0$ such that,
if we write $a^{ab}_0(y) := a^{ab}(y,z)$, then for every
$\xi = (\xi_0,\ldots, \xi_{n-1})\in \R^n$,
\[
c a^{ab}_0(y) \xi_a\xi_b \le a^{ab}(y,z)\xi_a\xi_b\le c^{-1} a^{ab}_0(y) \xi_a\xi_b\qquad\mbox{ for all }(y,z)\in [0,T_1]\times \Gamma^0\times (-\delta_1,\delta_1) .
\]
Then noting that $\pp_b \pp_z w_\ve (y,z)
= -\frac 1 \ve w''(\frac z\ve - h(y)) \pp_b h(y)
= -\ve \, \pp_{zz}w_\ve\,  \pp_b h$,
\begin{align*}
I_{1,1}&\ge \frac c2\int_\R\int_{\Gamma_0}a^{ab}_0(y)\left[ \pp_a\bar\vp^\perp \, \pp_b\bar\vp^\perp
+ \pp_a \gamma\, \pp_b \gamma (\pp_z w_\ve)^2 + \ve^2 \gamma^2 \, (\pp_{zz}w_\ve)^2\,\pp_a h\, \pp_b h\right] \omegazero(y')dy' dz\\
&\qquad  +
c\int_\R\int_{\Gamma_0} a^{ab}_0(y)[
\pp_a\bar \vp^\perp \pp_b\, \gamma\, \pp_z  w_\ve
-
\ve \, \pp_a\bar \vp^\perp  \gamma \,\pp_{zz} w_\ve\, \pp_b h ]\omegazero(y')dy' dz\\
&\qquad -
c\int_{\Gamma_0}  \int_\R a^{ab}_0\, \gamma \, \pp_a\gamma\pp_b h
 \ve  \, \pp_z w_\ve \,  \pp_{zz} w_\ve\,  \omegazero(y')dy' \,  dz
\end{align*}
The last term vanishes because $w$ is odd.
In any coordinate chart we can differentiate the orthogonality condition \eqref{ortho} to
find that
\[
0 = \pp_a \int_\R
\bar \vp^\perp \pp_z  w_\ve \, dz =
\int_\R
\pp_a \bar \vp^\perp \pp_z  w_\ve \, dz  -
\int_\R \ve \bar \vp^\perp \pp_{zz}  w_\ve\pp_a h \, dz .
\]
Using this we can rewrite the middle term as
\[
-c\int_{\Gamma_0} a^{ab}_0(y)\int_\R \ve\,  \pp_{zz}  w_\ve \pp_b h
(\bar \vp^\perp \pp_a\, \gamma\,
+
 \pp_a\bar \vp^\perp  \gamma ) \omegazero(y')dy' dz.
\]
For any $f$, we will write $|D_y f|^2 = a^{ab}\pp_a f\, \pp_b f$.
Since  $\int_\R (\pp_{zz}w_\ve)^2\,dz \le C\ve^{-3}$ and $|D_y h|\le C\ve$,
the above is bounded in absolute value by
\[
C\ve \int_{\Gamma^0}\int_\R (\bar\vp^\perp)^2 + |D_y\bar\vp^\perp|^2
 \omegazero(y')dy' dz
+C \int_{\Gamma^0} \gamma^2 + |D_y\gamma|^2
 \omegazero(y')dy' .
\]
Clearly $|D_y\bar \vp^\perp|^2 + (\pp_z\bar \vp^\perp)^2 \approx |D\bar\vp^\perp|^2 = (\pp_t \bar\vp^\perp)^2 + |\nabla_x \bar\vp^\perp|^2$, so we may combine the above
terms from $I_{1,1}$ with other terms estimated in   \eqref{I12.est}, \eqref{I112.est}, \eqref{I121},
to find that for $\ve$ sufficiently small,
\be\label{I1.est}
\begin{aligned}
I_1
&\ge -C\ve^2 I_2 - \frac C \ve \int_{{\Gamma^0}}\gamma^2  \, \omegazero(y')\,dy'
+
\frac c \ve \int_{{\Gamma^0}} |D_y\gamma|^2 \
  \omegazero(y')\,dy'
\\
&\qquad +
c \int_{{\Gamma^0}}\int_{\R} |D\bar \vp^\perp|^2 + \frac 1 {\ve^2}(\bar\vp^\perp)^2
 \, \omegazero(y')\,dz\,dy' \ .
\end{aligned}
\ee

In view of this and \eqref{I2.est},
we may fix a particular constant $C$ such that if we define $E(s)$
as in \eqref{Es.def}, which we recall is
\[
E(s) := E^{nr}_\ve(s) + E^{far}_\ve(s) +  \frac C
\ve  \int_{{\Gamma^0}}\gamma^2(s, y') \omegazero(y') dy' ,
\]
then for every $s\in [0,T_1]$, after adjusting $c$ we have
\begin{align}
E(s) &\ge
\frac c \ve \int_{{\Gamma^0}}\left(  \gamma^2 +  |D_y\gamma|^2 \right) \omegazero(y')dy' \Big|_{y_0=s}
+
c \int_{\Sigma_s}(1 - \chi_1^2)
\left( |D\vp|^2 +
 \frac 1 {\ve^2}  \vp^2 \right)
\omega_s
\nonumber\\
&\qquad\qquad
c \int_{{\Gamma^0}}\int_{\R} |D\bar \vp^\perp|^2 + \frac 1 {\ve^2}(\bar\vp^\perp)^2
 \, \omegazero(y')\,dz\,dy' \ .
\label{Et.est}
\end{align}
for $\omega_s := \omega^{nr}_s+ \omega^{far}_s$.
(We have extended $\chi_1$ to a function defined on all of $\Sigma_s$,
and vanishing on $\Sigma_s^{far}$.)

{\bf 6}.
We now want to show that
\[
\frac d{ds}E(s) \le C E(s).
\]
We first consider terms arising from $\frac d{ds}E^{nr}_\ve(s)$.
We rewrite the localized energy identity \eqref{Enear1} as
\[
\frac{d}{d s}E^{nr}_\ve(s) =
\int_{\Sigma^{nr}_s} \FF \, \omega^{nr}_s
-\int_{\delta_1}^{\delta_1}\int_{\Gamma^0}
 g^{n\b}
\frac {\pp \vp}{\pp y_\b}  \frac {\pp\vp }{\pp y_0}
(\chi^{nr})'(z)\omegazero(y') dy' dz \Big|_{y_0=s}
\]
where
\[
\FF :=
 -\frac {\pp }{\pp y_0}(  f'(u_\ve^*)) \frac{\vp^2} {2\ve^2}  +
b^\b
 \frac {\pp \vp}{\pp y_\b}  \frac {\pp\vp }{\pp y_0}
+
  \frac 1 2 (\pp_0g^{\a\b})
\frac {\pp \vp}{\pp y_\a}  \frac {\pp\vp }{\pp y_\b}
-\, \eta \, \frac {\pp\vp }{\pp y_0}.
\]
Recall that  $\chi_1 = 0$ in $\mbox{supp}(\chi^{nr})' \subset \{(y,z) : 2r_1\le z\le 4r_1\}$.
It follows that
\begin{align*}
\left|
\int_{\delta_1}^{\delta_1}\int_{\Gamma^0}
 g^{n\b}
\frac {\pp \vp}{\pp y_\b}  \frac {\pp\vp }{\pp y_0}
(\chi^{nr})'(z)\omegazero(y') dy' dz
\right|_{y_0=s}
\le
C \int_{\Sigma_s}(1 - \chi_1^2)
|D\vp|^2 \omega_s \nonumber
\le  CE(s).
\end{align*}
Straightforward estimates show that
\[
|\FF| \le C \left( |D\vp|^2 + \frac 1 {\ve^2}\vp^2\right) + C \eta^2 \mbox{ on the support of }1-\chi_1^2.
\]
(In particular, the construction of $u_\ve^*$ implies that $\partial_{0} f'(u_\ve^*) = 0$
on this set.) We use \eqref{Et.est}  to deduce that
\begin{equation}\label{Enr.est1}
\frac{d}{ds}E^{nr}_\ve(s) =
\int_{\Sigma^{nr}_s} \FF\  \chi_1^2(z) \omega^{nr}_s
+ C \int_{\Sigma^{nr}_s} (1-\chi_1^2)\eta^2 \omega^{nr}_s+ C E(s).
\end{equation}
We now consider various terms in the first integral  on the right-hand side above.
First, again writing $\chi_1 \vp = \bar \vp = \bar \vp^\perp +\gamma \,\pp_zw_\ve$, we decompose
\[
 \int_{\Sigma^{nr}_s}
\pp_0(  f'(u_\ve^*)) \frac{\vp^2} {2\ve^2}  \chi_1^2  \omega^{nr}_s
= I_{3,1} + I_{3,2}
\]
where
\begin{align*}
I_{3,1} &= \frac 1{2\ve^2} \int_{\Sigma^{nr}_s}\pp_0(  f'(u_\ve^*))[(\bar \vp^\perp)^2 +2\bar \vp^\perp  \gamma\,\pp_zw_\ve] \omega^{nr}_s,
\\ 
I_{3,2} &= \frac 1{2\ve^2} \int_{\Sigma^{nr}_s} \pp_0(  f'(u_\ve^*))
\gamma^2(\pp_zw_\ve)^2
\omega^{nr}_s .
\end{align*}
It follows from \eqref{ustar.recall}, \eqref{pico1} that
\be
\pp_0(  f'(u_\ve^*)) = (f''(w_\ve) + O(\ve^2))(\ve \pp_z w_\ve\, \pp_0 h + \pp_0\phi)
= f''(w_\ve) \ve \pp_z w_\ve \, \pp_0 h + O(\ve^2).
\label{pp0f}\ee
Moreover, $|f''(w_\ve) \ve \pp_z w_\ve \, \pp_0 h| \le C\ve$.
It follows that
\begin{align*}
|I_{3,1}|&
\le \frac C \ve\int_{\Sigma^{nr}_s}(\bar \vp^\perp)^2 + |\bar \vp^\perp| \ |\gamma \pp_z w_\ve|) \omega^{nr}_s
\le  C\int_{\Sigma^{nr}_s}\frac {(\bar \vp^\perp)^2}{\ve^2} + \gamma^2 (\pp_z w_\ve )^2 \omega^{nr}_s
\\
&\le \frac C{\ve^2}\int_{\Sigma^{nr}_s}(\bar \vp^\perp)^2 \omega^{nr}_s + \frac C \ve\int_{\Gamma^0}\gamma^2(s,y') \, \omega^0(y')dy' \\
&\overset{\eqref{Et.est}}\le CE(s) .
\end{align*}
We next use \eqref{pp0f} to write
\begin{align*}
I_{3,2}
&=\frac 1{2\ve} \int_{\Sigma^{nr}_s } [ f''(w_\ve) \pp_zw_\ve \, \pp_0h + O(\ve)] \gamma^2 (\pp_z w_\ve)^2 \, \omega^{nr}_s \, .
\end{align*}
Since $f$ and $w$ are odd, and $w_\ve = w(\frac z\ve - h(y))$, we find that
$z\mapsto f''(w_\ve) (\pp_z w_\ve)^3$ is odd, modulo a translation by $h(y) = O(\ve)$.
Since it decays exponentially, its integral over $z\in(-r_2,r_2)$ is thus exponentially small.
We then easily deduce that
\[
|I_{3,2}| \le \frac C\ve\int_{\Gamma^0}\gamma^2(s,y') \, \omega^0(y')dy'
\]
and hence that
\be
\left| \int_{\Sigma^{nr}_s}
\pp_0(  f'(u_\ve^*)) \frac{\vp^2} {2\ve^2}  \chi_1^2  \, \omega^{nr}_s\right|
\le C E(s) .
\label{I3.est}\ee


We next consider the convection term.
Recall that by definition,
\begin{align*}
b^\b :=\frac \omegazero { \sqrt {|\det g|}}g^{\a\b}\frac \pp {\pp y_\a} \left(\frac{ \sqrt{|\det g|}}\omegazero  \right).
\end{align*}
In particular, since $g^{\a n} = \delta^{\a n}$ in supp$(\chi_1)$
and  $\omegazero$ is independent of $z$,
\[
b^n(y,z) =
\frac{1} {\sqrt {|\det g(y,z)|  }}
\frac {\pp}{\pp z}  \sqrt{|\det g(y,z)|}  \ ,
\]
and thus it follows from \eqref{dz.gnn} that
\[
|b^n(y,z)|\le C|z| \qquad\mbox{ in }\mbox{supp}(\chi_1).
\]
Therefore
\begin{align*}
\left| b^\gamma \frac{\pp \vp}{\pp y_\gamma} \frac{\pp  \vp}{\pp y_0}\right| \chi_1^2
&\le  C \left( a^{ab} \frac{\pp \bar \vp}{\pp y_a}\frac{\pp \bar \vp}{\pp y_b}  + z^2 (\frac{\pp \vp}{\pp z})^2\chi_1^2\right)\\
&=  C \left( a^{ab} \frac{\pp \bar \vp}{\pp y_a}\frac{\pp \bar \vp}{\pp y_b}  + z^2 (\frac{\pp \bar \vp}{\pp z})^2 -  z^2 \chi_1'^2 \vp^2 -  z^2 \chi_1 \chi_1' \frac{\pp}{\pp z} (\vp^2)\right) .
\end{align*}
We again write $\bar \vp = \bar \vp^\perp + \gamma(y) w_\ve'(z)$.
Using the fact that $\int_\R z^2 w_\ve''^2 dz  = \frac C \ve $,
and arguing as in the estimate of $I_{1,3}$ above, we find that
\[
\left| \int_{\Sigma^{nr}_s}
b^\gamma \frac{\pp \vp}{\pp y_\gamma} \frac{\pp  \vp}{\pp y_0} \chi_1^2\, \omega^{nr}_s \right|
\overset {\eqref{Et.est}}\le C E(s)\ .
\]
Next, again because $g^{\alpha n}=\delta^{\a n}$
in $\mbox{supp}(\chi_1)$,
it is clear that
\[
\left| \int_{\Sigma^{nr}_s}
  (\pp_0g^{\a\b})
\frac {\pp \vp}{\pp y_\a}  \frac {\pp\vp }{\pp y_\b} \chi_1^2(z) \omega^{nr}_s \right|
=
\left| \int_{\Sigma^{nr}_s}
  (\pp_0g^{ab })
\frac {\pp \bar\vp}{\pp y_a}  \frac {\pp\bar \vp }{\pp y_b}  \omega^{nr}_s \right|
\overset {\eqref{Et.est}} \le C E(s) .
\]
Finally, since $\frac\pp{\pp y_0}\bar\vp =\frac\pp{\pp y_0}(\chi_1\vp) = \chi_1
\frac\pp{\pp y_0}\vp$,
\begin{align*}
\left| \int_{\Sigma^{nr}_s}
\frac{\pp\vp}{\pp y_0}\eta  \chi_1^2 \ \omega^{nr}_s \right|
&\le
\left| \int_{\Sigma^{nr}_s}
(\frac{\pp\bar \vp}{\pp y_0})^2 \  + \chi_1^2\eta^2  \, \omega^{nr}_s \right|
\\
&
\overset {\eqref{Et.est}}\le C E(s) +
 C \int_{\Sigma^{nr}_s} \chi_1^2\eta^2 \omega^{nr}_s
 \end{align*}
Putting the above estimates into \eqref{Enr.est1}, we
obtain
\[ 
\frac d{ds} E^{nr}_\ve(s) \le    C E(s)+  C \int_{\Sigma^{nr}_s} \eta^2 \omega^{nr}_s.
\] 

{\bf 7}.
Similar but  easier arguments, using \eqref{Efar1} and \eqref{Et.est}, lead to the estimate
\[ 
\frac d{ds} E^{far}_\ve(s) \le    C E(s)+  C \int_{\Sigma^{far}_s} \eta^2 \omega^{far}_s.
\] 
The point is that  $\mbox{supp}(D\chi^{far})\subset \{ \chi_1^2 = 0\}$.
Thus  terms in \eqref{Efar1} containing derivatives of $\chi^{far}$
are easily estimated by
$\int_{\Sigma_s}(1 - \chi_1^2)
\left( |D\vp|^2 +
 \frac 1 {\ve^2}  \vp^2 \right)
\omega_s \le CE(s)$.

Similarly,
\[ 
\frac d {d s}  \frac 1 \ve \int_{\Gamma^0} \gamma^2 \omegazero(y')dy'
\le
  \frac 1 \ve \int_{\Gamma^0} ( \gamma^2+  (\pp_{0}\gamma)^2) \omegazero(y')dy' \le C E(s).
\] 
Combining the last three inequalities,
we conclude that
\begin{align}\label{E2}
\frac  d{ds} E(s) \le C E(s) + C \int_{\Sigma_s} \eta^2\  \omega_s
\end{align}
Then it follows from Gr\"onwall's inequality that
\begin{equation}\label{E3}
E(s) \le e^{Cs} E(0) +
C \int_0^s \int_{\Sigma_s} e^{C(s-\sigma)} \eta^2 \, \omega_\sigma  \, d\sigma
\end{equation}
for all $0\le s < T_1$.

{\bf 8.} The conclusion \eqref{crude.linear} of the Proposition follows from \eqref{E3}.
We sketch the straightforward verification.

Recall that  the $n$-form $\omega_s$ is uniformly comparable to the induced
Euclidean $n$-dimensional area in $\Sigma_s$, a fact that we will use repeatedly and without further mention.
Thus for example  it is immediate that
\[
 \int_0^s \int_{\Sigma_\sigma} e^{C(s-\sigma)} \eta^2 \, \omega_\sigma  \, d\sigma
\le C \int_0^s\left(\int_{\Sigma_s} \eta^2 \right) d\sigma \qquad \mbox{ for all }s\in [0,T_1],
\]
where on the right-hand side, we implicity integrate with respect to the Euclidean area, as in
\eqref{crude.linear}.
Next we claim that the left-hand side of \eqref{crude.linear} is bounded by $C E(s)$.
Toward this end, first note from \eqref{split.L2} that
\begin{equation}\label{L2.full}
\int_{\Sigma_s}\vp^2 \omega_s = \int_{\Sigma_s}(1-\chi_1^2) \vp^2  \omega_s +
\int_{\Gamma^0} \int_\R( \bar \vp^\perp)^2 \omegazero(y')\,dz\, dy'\  +\frac \Xi \ve \int_{\Gamma^0}\gamma^2 \omegazero(y')dy'.
\ee
Thus \eqref{Et.est} implies that $\int_{\Sigma_s}\vp^2 \omega_s \le C E(s)$.

Next, we write
$|D\vp|^2 = (1-\chi_1^2)|D\vp|^2 + \chi_1^2 |D\vp|^2$. The first term is immediately controlled $C E(s)$,
due to \eqref{Et.est}. For the second term we may compute in modified Fermi coordinates, in which we
have
\begin{align*}
\chi_1^2|D \vp|^2
&\le C \chi_1^2 [ a^{ab} \pp_a\vp \, \pp_b\vp + (\pp_z\vp)^2]
\\
&\le
C\left[ a^{ab} \pp_a\bar \vp \, \pp_b\bar\vp + (\pp_z\bar \vp)^2 + \chi_1(z)\chi_1''(z)\vp^2 - \pp_z(\chi_1\chi_1' \vp^2) \right]
\end{align*}
Note that $\int_{\Sigma^{nr}_s} a^{ab} \pp_a\bar \vp \, \pp_b\bar\vp \omega^{nr}_s$ is exactly the
term $I_{1,1}$ that appeared above in the lower bound for $E(s)$. There it was convenient to split it into
several pieces (from which we obtained separate control over $(\pp_0\gamma)^2$, needed above),
but if we keep that term as it is, then our earlier estimates of all the other contributions show that  $I_{1,1} \le C E(s)$.
The terms involving derivatives of $\chi_1$ are handled as in our estimate of $I_{1,3}$.
Finally, we write $(\pp_z\bar \vp)^2 = (\pp_z\bar \vp^\perp + \gamma \pp_{zz}w_\ve)^2
\le 2 (\pp_z\bar \vp^\perp)^2+ 2 \gamma^2 (\pp_{zz}w_\ve)^2$.
By integrating and combining with the above estimates, we finally conclude
that $\ve^2\int_{\Sigma_s}\chi_1^2 |D\vp|^2 \le C E(s)$, with the loss of a factor of $\ve^2$
coming from the term $2 \gamma^2(\pp_{zz}w_\ve)^2$.

To finish the verification of \eqref{crude.linear}, we must check that
\begin{align}
E(0)
&=
\int_{\Sigma^{nr}_0} e^{nr}_\ve(\vp)\omega^{nr}_0
+
\int_{\Sigma^{far}_0} e^{far}_\ve(\vp)\omega^{far}_0
+\frac C \ve \int_{\Gamma_0}\gamma^2(0,y')\omegazero(y')dy'
\nonumber \\
&\le C\int_{\R^n}\left[ |\nabla_x \vp_0|^2 +  |\vp_1|^2+ \frac 1{\ve^2}\vp_0^2 \right] dx .
\label{rhs1}\end{align}
Indeed, it is immediate from the definition \eqref{enr.def} that
\[
e^{nr}_\ve(\vp)(y,z) \le C\left[  (\pp_t \vp)^2 + |\nabla_x \vp|^2 + \ve^{-2}\vp^2\right](\Phi(y,z))
\]
Since $\Phi(y,z)\in \{0\}\times \R^n$ when $y_0=0$, see \eqref{y0=tbis}, the initial condition $(\vp, \pp_t\vp)|_{t=0} = (\vp_0,\vp_1)$
implies that
\[
\int_{\Sigma^{nr}_0} e^{nr}_\ve(\vp)\omega^{nr}_0\le C\int_{\R^n}\left[ |\nabla_x \vp_0|^2 +  |\vp_1|^2+ \frac 1{\ve^2}\vp_0^2 \right] dx
\]
The corresponding estimate for $e_\ve^{far}$ on $\Sigma^{far}_0$ is immediate.
We conclude from these facts and \eqref{L2.full} that \eqref{rhs1} holds.

\end{proof}

\begin{remark}\label{linear.sharper}
In view of \eqref{Et.est}, it follows from \eqref{E3} that
\[ 
\begin{aligned}
&\frac c \ve \int_{{\Gamma^0}}\left(  \gamma^2 +  |D_y\gamma|^2 \right) \omegazero(y')dy' \Big|_{y_0=s}
+
c \int_{\Sigma_s}(1 - \chi_1^2)
\left( |D\vp|^2 +
 \frac 1 {\ve^2}  \vp^2 \right)
\omega_s
\\
&
\qquad\qquad \qquad\qquad \qquad\qquad +
c \int_{{\Gamma^0}}\int_{\R} |D\bar \vp^\perp|^2 + \frac 1 {\ve^2}(\bar\vp^\perp)^2
 \, \omegazero(y')\,dz\,dy' \\
&\qquad \qquad \le
C \int_0^s\Big( \int_{\Sigma_\sigma} \eta^2 \Big) d\sigma +
C\int_{\R^n}\left[ |\nabla_x \vp_0|^2 +  |\vp_1|^2+ \frac 1{\ve^2}\vp_0^2 \right] dx .
\end{aligned}
\]
This is considerably stronger  than \eqref{crude.linear}
\end{remark}

\subsection*{Higher order estimates  }

We need estimates similar to \eqref{crude.linear} for higher order space derivatives.
It is convenient to introducing the $L^2$-$L^\infty$ norm for functions $\eta$ defined on $\Sigma$,
$$
\|\eta \|_{L^\infty L^2(\Sigma)}\ := \ \sup_{0\le s\le T_1} \| \eta\|_{L^2(\Sigma_s)}.
$$
Then from \eqref{crude.linear} we get the  $L^{\infty}$-$ H^1$-estimate for the solution $\vp$ of Problem \equ{linear.eqn}
\be\label{sti}
\begin{aligned}
\|\vp \|_{L^\infty L^2(\Sigma)}  + \ve \|\vp_t \|_{L^\infty L^2(\Sigma)}+  \ve \|D_x\vp \|_{L^\infty L^2(\Sigma)}\ \le  \  \\
 C \Big[
\|\eta \|_{L^\infty L^2(\Sigma)}+
 \frac 1{\ve}  \big ( \, \|\vp_0\|_ {L^2(\R^n)}  +  \ve \|D_x \vp_0\| _ {L^2(\R^n)}  +  \ve \|\vp_1\|_ {L^2(\R^n)}  \big ) \Big].
 \end{aligned}
\ee
Assuming further smoothness on initial data and the right hand side we can derive higher order estimates as follows.
Let us differentiate twice the equation. We write $D_i = \pp_{x_i}$, $D_{ij} =\pp^2_{x_i x_j}$. Then we get
$$
\begin{aligned}
L [ D_{ij} \vp]  \  =& \  D_{ij} \eta  -  \ve^{-2}  \vp D_{ij}( f'(u_\ve^*)) -  \ve^{-2} ( D_i\vp D_j( f'(u_\ve^*)) + D_j\vp D_i( f'(u_\ve^*)) )\inn \Sigma, \\
(D_{ij}\vp, \pp_tD_{ij}\vp)\Big|_{t=0}\ =&\  (D_{ij}\vp_0, D_{ij} \vp_1)  .
\end{aligned}
$$
Using estimate \equ{sti} and the facts  $D_{i}( f'(u_\ve^*))= O(\ve^{-1})$, $D_{ij}( f'(u_\ve^*))=O(\ve^{-2})$ we get
\be\label{bound1} \begin{aligned}
 &  \| D_x^2 \vp \|_{L^\infty L^2(\Sigma)} + \ve \|D_x^2\vp_t \|_{L^\infty L^2(\Sigma)}+  \ve \|D^3_x\vp \|_{L^\infty L^2(\Sigma)}\\
 &\le   \ \frac C{\ve^4}  \Big [
\| \eta \|_{L^\infty L^2(\Sigma)}  + \ve^4  \|D^2_x \eta \|_{L^\infty L^2(\Sigma)} \big ]   \\ & \quad +
   \frac C{\ve^5} \Big[
 \,  \|\vp_0\|_ {L^2(\R^n)} + \ve \|D_x \vp_0\| _ {L^2(\R^n)}+  \ve \|\vp_1\|_ {L^2(\R^n)} \big ]   \\  & \quad   +   \frac C\ve \Big [ \|D^2_x \vp_0\| _ {L^2(\R^n)} +  \ve \|D^3_x \vp_0\| _ {L^2(\R^n)}+ \ve \|D_x^2 \vp_1\| _ {L^2(\R^n)} \Big] .
 \end{aligned}\ee
Let us consider for $m\ge 0$  the following $L^\infty$-$H^m$ norm
$$
\|\eta\|_{L^\infty H^m(\Sigma) }  :=    \sum_{j=0}^m \|D^j\eta\|_{L^\infty L^2(\Sigma) } .
$$
Then from  \equ{bound1} and interpolating with bound \equ{sti}, the following estimate readily follows:
\be \label{bound2}
\|\vp\|_{L^\infty H^3(\Sigma) }\ \le\  \frac C{ \ve^5} \Big [ \|\eta\|_{L^\infty H^2(\Sigma) } +  \|\vp_1\|_{H^2(\R^n) } +  \ve^{-1}\|\vp_0\|_{H^3(\R^n) }\Big ]
\ee
An induction argument (differentiating an even number of times $m$) yields
$$ \begin{aligned}
 &  \| D_x^m \vp \|_{L^\infty L^2(\Sigma)} + \ve \|D_x^m\vp_t \|_{L^\infty L^2(\Sigma)}+  \ve \|D^{m+1}_x\vp \|_{L^\infty L^2(\Sigma)}\\
 &\le   \ \frac C{\ve^{2m}}  \Big [
\| \eta \|_{L^\infty L^2(\Sigma)}  + \ve^{2m}  \|D^m_x \eta \|_{L^\infty L^2(\Sigma)} \big ]   \\ & \quad +
   \frac C{\ve^{2m+1}} \Big[
 \,  \|\vp_0\|_ {L^2(\R^n)} + \ve \|D_x \vp_0\| _ {L^2(\R^n)}+  \ve \|\vp_1\|_ {L^2(\R^n)} \big ]   \\  & \quad   +   \frac C\ve \Big [ \|D^{m}_x \vp_0\| _ {L^2(\R^n)} +  \ve \|D^{m+1}_x \vp_0\| _ {L^2(\R^n)}+ \ve \|D_x^m \vp_1\| _ {L^2(\R^n)} \Big] .
 \end{aligned}$$
As in \equ{bound2} we finally find the estimate

\begin{lemma}
The solution $\vp$ of Equation $\equ{linear.eqn}$
satisfies the estimate
\be \label{bound22}
\|\vp\|_{L^\infty H^{m+1}(\Sigma) }\ \le\  \frac C{ \ve^{2m+1}} \Big [ \|\eta\|_{L^\infty H^m(\Sigma) } +  \|\vp_1\|_{H^m(\R^n) } +  \ve^{-1}\|\vp_0\|_{H^{m+1}(\R^n) }\Big ]
\ee
for each even integer $m$.
\end{lemma}

\section{The proof of Theorem \ref{teo1}}
Now we have all the ingredients to proceed to the proof of Theorem \ref{teo1}.  We look for a solution to Problem
$S(u)=0$  close the approximation
$u_\ve^*$  given by \equ{uestar}, where the number $k$ will be chosen sufficiently large.  We look for a solution of the form
$$
u(x,t) =  u_\ve^*(x,t) + \vp(x,t) .
$$
In terms of $\vp$ the equation becomes
\be
L_\ve [\vp]  +   \ve^{-2}S(u_*)  +  \ve^{-2} N(\vp,x,t) = 0\inn \Sigma
\label{abc}\ee
where $$N(\vp,x,t) =  f(u_\ve^*(x,t)+\vp)  - f'(u_\ve^*(x,t))\vp -f(u_\ve^*(x,t))$$
$$ L_\ve [\vp]  = \Box \vp  + \ve^{-2} f'(u_\ve^*)\vp .$$

\medskip
Let us consider the unique solution $\vp =  \TT [\eta]$ of the linear problem
\begin{equation}
L_\ve [\vp] + \eta=0 \ \ \mbox{ in } \Sigma,
\qquad\qquad
(\vp, \pp_t\vp)\Big|_{t=0} = (0,0)  .
\label{linear.eqnB}\end{equation}
which we have estimated in Proposition \ref{prop.linear}.
Problem \equ{abc}  with initial data $(\vp, \pp_t\vp)\Big|_{t=0} = (0,0) $ can then be written as the fixed point problem
\be
\vp =  \TT [   \ve^{-2} S(u_\ve^*)  + \ve^{-2} N(\vp, \cdot ) ]  =: \mathcal M (\vp), \quad \vp \in \BB
\label{ss1}\ee
We will solve this problem by contraction mapping principle in a suitable space $\BB$ of small functions defined on $\Sigma$.
We consider the Banach space $L^\infty H^m(\Sigma)$ endowed with its natural norm and consider the region
$$
\BB  =  \{ \vp\in L^\infty H^m(\Sigma)\ /\ \|\vp\|_{L^\infty H^m(\Sigma) } \le \ve^{\frac k2} \}
$$
where $k$ is the number in the definition of $u_\ve^*$ in \equ{uestar}.
 Let us fix a number  $m>n/2$. Among other consequences, this implies that $L^\infty H^m(\Sigma) $ is embedded into $L^\infty (\Sigma)$ and
\be\label{cotaHm}
\| \vp \psi  \|_{L^\infty H^m(\Sigma)} \le C \| \vp\|_{{L^\infty H^m(\Sigma)}}\|\psi \|_{L^\infty H^m(\Sigma)}.
\ee

\medskip
From Proposition \ref{prop.as} (specifically bound \equ{pico2}), together with the fact that errors are exponentially small in $\ve$ where the cut-off is not constant,
we find that
$$
\|  \ve^{-2} S(u_\ve^*) \|_{L^\infty H^m(\Sigma) }\ \le\  C\ve^k.
$$
Next, we claim that for $\vp, \tilde \vp \in \BB$,
\be
\| N(\vp, \cdot) - N(\tilde \vp, \cdot)\|_{L^\infty H^m(\Sigma)} \le   C \ve^{\frac k 2 -m}\|\vp - \tilde \vp\|_{L^\infty H^m(\Sigma)}
\label{nest}\ee
To prove this, observe that
\[
N(\vp)-N(\tilde \vp) =\left(\int_0^1 \left[  f'(u_\ve^* + \sigma \vp +(1-\sigma)\tilde \vp)  - f'(u_\ve^*)\right] d\sigma\right) \left( \vp - \tilde \vp\right).
\]
In view of \equ{cotaHm}, to establish \equ{nest}
it suffices to observe that
\[
\|f'(u_\ve^*(t,\cdot) + \psi)  - f'(u_\ve^*) \|_{{L^\infty H^m(\Sigma)}}\ \le\  C \ve^{-m}\|\psi\|_{L^\infty H^m(\Sigma)}\ \le\     C \ve^{\frac k2-m},
\qquad\mbox{ for all }\psi \in \BB .
\]
which follows from a direct computation using Leibnitz rule, Sobolev embedding and the fact that $D^m _x u_\ve^* = O(\ve^{-m})$.

\medskip
At this point we fix a number $k$ with $k> 6m +2$. Let us consider the operator $\mathcal M[\vp]$ defined on $\BB$ in formula \equ{ss1}.
Estimates \equ{bound22} and \equ{nest} lead to
$$
\| \mathcal M[\vp] - \mathcal M[\ttt \vp]\|_{L^\infty H^m (\Sigma)}\ \le \  C \ve^{\frac k 2 -3m -1 }\|\vp - \tilde \vp\|_{L^\infty H^m(\Sigma)} \foral \vp,\ttt\vp\in \BB .
$$
and
$$
\| \mathcal M[0]|_{L^\infty H^m (\Sigma)}\ \le\ C\ve^{k-2m-1} .
$$
It follows that for all sufficiently small $\ve$ we get that $\mathcal M(\BB)\subset \BB$ and that $\mathcal M$ is a contraction mapping in $\BB$. Hence Problem \equ{ss1} has a unique solution. The conclusion of Theorem \ref{teo1} readily follows.  \qed

\medskip
This proof applies equally well to yield the stability assertion made at the end of \S \ref{Statement} by just considering the
operator $\TT$ involving sufficiently small initial data.

\numberwithin{theorem}{section}
\numberwithin{prop}{section}
\numberwithin{lemma}{section}
\numberwithin{remark}{section}
\appendix

\section{Modified Fermi Coordinates}\label{AppA}
\label{add}

In this appendix we present the proof of Lemma \ref{L.modferm}.

\begin{proof}
For the proof, we will denote the modified Fermi coordinates as
\[
(\py, \pz) = (\py_0, \py', \pz) = (\py_0,\ldots, \py_{n-1},\pz)
\]
(denoted $(y,z)$ in the statement of the lemma and elsewhere in this paper),
and we will reserve $(y,z)$ for Fermi coordinates associated to
a canonical local parametrization $(y,z)\in [0,T]\times V_l\times(-\delta, \delta)\mapsto Y_l(z)+z\nu(y)$, as constructed in Section \ref{coords}.
We specify that the relationship between modified and Fermi coordinates has the form
\[
(y_0, y', z) = (y_0(\py, \pz), \py', \pz)
\]
where $y_0$ depends on $(\py,\pz)$ in a way to be described below.
Thus they are related to $(x,t)$ coordinates
via
\[
(x,t) = Y_l(y_0(\py, \pz), \py', \pz) + \pz \nu(y_0(\py, \pz) ,\py' ) =: \Phi_l(\py,\pz)
\]
for
\[
(\py,\pz)\in [0,T_1]\times V_l\times (-\delta_1,\delta_1)
\]
We will also write
\begin{align*}
\widetilde Y_a
&:= \frac{\partial \Phi}{\partial\py_a} 
\quad\mbox{ for }a=0,\ldots, n-1
\\
\widetilde Y_n
&:= \frac{\partial \Phi}{\partial\pz}  
\\
\tg_{\a\b}
&:= \metric{Y_\a,Y_\b},
\qquad \mbox{ for }\a, \b = 0, \ldots, n.
\end{align*}

\medskip

To define $y_0(\py, \pz)$, fix some $l$ and consider  $Y_l : \Lambda_l\to \Gamma$
as in \eqref{clp}.
We will often omit the subscript $l$, and we will write
\[
(t(y,z), x(y,z)) = Y(y) + z\nu(y)
\]
to indicate the dependence of $(x,t)$ on $(y,z)$.
Recall that
$Y_l$ is constructed to that
$t(y, 0) = t(y_0,y', 0) = y_0$, see \eqref{canonical.coords}, and hence $\frac{\partial t}{\partial y_0}(y, 0) = 1$ everywhere in $\Lambda_l$.
Thus the Implicit Function Theorem implies that
for any $T_1<T$ there exists $\delta_1<\delta$ and a function $\eta_0:[0,T_1]\times V_l\times (-\delta_1,\delta_1)\to [0,T]$ such that
\[
\eta_0(\py, 0 ) = \py_0, \qquad t(\eta_0(\py, \pz) , \py', \pz) = \py_0 \qquad\mbox{ everywhere in }[0,T_1]\times V_l\times (-\delta_1,\delta_1).
\]
Here we are implicitly using our assumption that
the velocity of $\Gamma$ vanishes at $t=0$.
For $Y\in \Gamma^0$ and $|z|<\delta$,
this implies that  $Y+z\nu(Y)$ belongs to $\{0\}\times \R^n$,
and hence that
$t(0,\py', \pz)= 0$
for all $(\py',\pz)\in V_l\times (-\delta, \delta)$.
It is this property that allows us to extend the domain of $\eta_0$ all the way to $\{y_0=0\}$,
and it implies that $\eta_0(0,\py', \pz) = 0$
for all $(\py', \pz)$.

We can choose $\delta_1$ such that the above properties hold for all $\Lambda_l$,
$l=1,\ldots, m$.

We will take $y_0(\py, \pz)$ to have the form
\[
y_0(\py, \pz) = \chi_0(\pz)  \py_0 + (1- \chi_0(\pz)) \eta_0(\py, \pz)
\]
where $\chi_0$ will be specified below in the proof of \eqref{gab.signs}.
It will be the case that
\[
\chi_0(\pz) = 1 \mbox{ if }|\pz|< r_2, \qquad
\chi_0(\pz) = 0 \mbox{ if }|\pz| > r_1,
\]
for  $0<r_2< r_1<\delta_1$ also to be fixed below.
We will  require that $r_1, r_2, \chi_0$ are chosen
uniformly for $l=1,\ldots, m$, so that \eqref{ind.of.l} holds.

It is immediate that  $(\py,\pz) = (y,z)$, and hence $\tg_{\a\b} = g_{\a\b}$ for $(\py,\pz)\in [0,T_1]\times V\times (-r_2,r_2)$.
Thus \eqref{ggl.gab}  follows from the corresponding properties
of Fermi coordinates, see \eqref{basic.fermi}.
Similarly, \eqref{dz.gnn} is a basic property of Fermi coordinates,
together with the fact that $\Gamma$ is minimal, see \eqref{mean} and \eqref{H}
Likewise, \eqref{y0=tbis} is a straightforward consequence of the definition of $y_0(\py, \pz)$.

It remains only to prove \eqref{gab.signs}.
To do this we note that
\[
\widetilde Y_0 = \frac\pp {\partial \py_0}(Y + \pz \nu) =
 \frac\pp {\partial y_0}(Y + \pz \nu) \frac{\pp y_0}{\pp \py_0} = Y_0  \frac{\pp y_0}{\pp \py_0}.
\]
and similarly
\[
\widetilde Y_i = Y_0 \frac{\partial y_0}{\partial \py_i} + Y_i \ \ \ \mbox{ for }i=1,\ldots,n
\]
where here and below,  we sometimes write $\pz$ as $\py_n$.
It follows that
\begin{align}
\tg_{00}&= (\frac{\pp y_0}{\pp \py_0})^2 g_{00},
\nonumber\\
\tg_{0i} = \tg_{i0}&=\frac {\pp y_0}{\pp \py_0}\frac {\pp y_0}{\pp \py_i} g_{00},
\label{gvsg}\\
\tg_{ij}&=\frac {\pp y_0}{\pp \py_i}\frac {\pp y_0}{\pp \py_j} g_{00} + g_{ij}
\nonumber
\end{align}
for $i,j = 1,\ldots, n$.

The fact that
$\tg_{00}<0$ everywhere now follows from \eqref{gvsg} and the corresponding property of Fermi coordinates.

We next prove that
$ \left[ \begin{array}{r}  \tg_{ij}
\end{array}
\right]_{i,j=1}^{n}$
is positive definite.
For $|\pz|\le r_2$
this follows from standard properties of Fermi coordinates.
For $|\pz| \ge  r_1$ it is also straightforward. Indeed, in this set, for every
fixed $\py_0$, the map $(\py',\pz)\to Y+\pz \nu$
is just a parametrization of a portion of the hypersurface
$\{ y_0\}\times \R^n$, on which the
induced metric is simply the Euclidean metric. So
in this set, the metric tensor $ \left[ \begin{array}{r}  \tg_{ij}
\end{array}
\right]_{i,j=1}^{n}$ is
just the Euclidean metric rewritten with respect to a new coordinate system.
Hence it is clearly positive definite.

We now consider $r_2<|z|<r_1$. We start with the main point which, it turns out,
is to fix
$r_1, r_2$ and $\chi_0$ so that $\tg_{nn} $ is bounded away from zero.
Since $g_{nn} = 1$, as recalled in the proof of \eqref{ggl.gab},
\begin{equation}\label{gnnn}
\tg_{nn}=  1 + (\frac{\pp y_0}{\pp \pz})^2 g_{00} \ .
\end{equation}
For fixed $\py = (\py_0,\py')$, consider the curve
\[
\pz\mapsto  Y(\eta_0(\py, \pz), \py',\pz) + \pz \nu(\eta_0(\py, \pz), \py',\pz)  =: \Phi_0(\py, \pz)
\]
and let $X$ denote the tangent vector $\pp \Phi_0 /\pp \pz$.
The definition of $\eta_0$ implies that the image of the curve is contained in
the hypersurface $\{\py_0\}\times \R^n$, and hence that $X$
is spacelike, or in other words that $\metric{X,X}>0$.
By compactness,  after poisbly shrinking $\delta_1$
there exists some $c>0$ such that
$\metric{X,X} \ge c$ everywhere in  $[0,T_1]\times V\times (-\delta_1,\delta_1)$.
Writing out this inequality in
coordinates, and again using the fact that $g_{nn}=1$, we obtain
\begin{equation}\label{hatgnn}
1+ (\frac{\pp \eta_0}{\pp \pz} (\py,\pz))^2\  g_{00}(\eta_0(\py,\pz),\py',\pz) \ge c .
\end{equation}
Next, since $\py_0 = \eta_0(\py_0,0) = \eta_0(\py,\pz) - \pz \partial_\pz\eta_0(\py,\pz) + O(\pz^2) $,
we use the definition of $y_0(\py, \pz)$ to compute
\begin{align*}
\frac{\pp y_0}{\pp \pz}(\py,\pz)
&=  \chi'_0(\pz)(\py_0-\eta_0(\py,\pz)) +(1-\chi_0(\pz))\frac{\pp \eta_0}{\pp \pz}(\py,\pz)\\
&= \frac{\pp \eta_0}{\pp \pz}(\py,\pz)( 1 - \chi_0(\pz) - \pz \chi_0'(\pz)) + O(\pz^2\chi_0'(\pz)) .
\end{align*}

We now take $\chi_0$ of the form
\[
\chi_0(\pz) = \begin{cases}
1&\mbox{ if }\pz\le r_2 = r_1^2/(1+ r_1) \\
r_1( \frac {r_1}{\pz}-1)&\mbox{ if }r_2\le \pz \le r_1\\
0&\mbox{ if }\pz\ge r_1
\end{cases}.
\]
for $r_1>0$ to be chosen below. (More precisely, we take $\chi_0$
to be a regularization of the function defined above, and satisfying essentially
the same estimates. But for simplicity we will compute with the function
defined above, which is merely Lipschitz.)
With this choice,
$ - \chi_0(\pz) - \pz \chi_0'(\pz) = r_1$ on the support of $\chi_0'$, so
\[
\frac{\pp y_0}{\pp \pz}(\py,\pz)
= \partial_\pz\eta_0(\py,\pz)(1+ r_1) + O(r_1^2).
\]
We also observe that
$|\eta_0(\py,\pz) - y_0(\py,\pz)|\le C|\pz|$,
because  $\eta_0(\py,0) = y_0(\py,0) = \py_0$.
It follows that
\[
g_{00}(y_0(\py,\pz), \py',\pz) = g_{00}(\eta_0(\py,\pz), \py',\pz) + O(|\pz|).
\]
By combining these with \eqref{gnnn}, \eqref{hatgnn}, we  find that
\[
\tg_{nn} (\py,\pz) \ge c  -  C r_1
\]
for $C$ depending on $\| g_{00}\|_{W^{1,\infty}}$ and $\|\partial_\pz\eta_0\|_{L^\infty}$.
It follows that
\begin{equation}\label{gnn.lbd}
\tg_{nn}\ge c/2 \qquad\mbox{ at all points where $r_2\le z \le r_1$.}
\end{equation}
for all sufficiently small choices of  $r_1$ (and hence $r_2$) in the definition of $\chi_0$.

We next remark that since $\eta_0(\py,0) = \py_0$ for all $\py$, it is clear that
$\frac{\partial \eta_0}{\partial \py_i}(\py,0) = 0$ for $i=1,\ldots, n-1$.
It follows that
\[
|\frac{\partial \eta_0}{\partial \py_i}(\py,\pz)|\le  C|\pz|\ \ \ \mbox{ for }i=1,\ldots, n-1,
 \qquad
\qquad
|\eta_0(\py,\pz) - \py_0|\le C|\pz|,
\]
everywhere in its domain, and hence that the same properties hold for $y_0(\py,\pz)$.
We then see from  \eqref{gvsg} that for $r_2\le |\pz|\le r_1$,
\begin{align*}
|\tg_{ij}- g_{ij}|  &\le Cr_1^2
\qquad\mbox{for $1\le i,j\le n-1$ }, \ \ \mbox{ and }
\\
|\tg_{in}| = |\tg_{ni}| &\le C r_1 \qquad\mbox{for $1\le i\le n-1$}.
\end{align*}
Since  $\left[ \begin{array}{r}  g_{ij}
\end{array}
\right]_{i,j=1}^{n-1}$ is positive definite,
we conclude from this and \eqref{gnn.lbd} that $ r_1$ may be chosen so that
$ \left[ \begin{array}{r}  \tg_{ij}
\end{array}
\right]_{i,j=1}^{n}$ is positive definite everywhere.

Finally, the facts that $\tg_{00}<0$ and $ \left[ \begin{array}{r}  \tg_{ij}
\end{array} \right]_{i,j=1}^n$ is positive definite
imply the same properties for $\tg^{00}$ and $ \left[ \begin{array}{r}  \tg^{ij}
\end{array}
\right]_{i,j=1}^n$.
This is a consequence of the general formula for the inverse of a matrix in block form
\[
\left(\begin{array}{ll}
a & b\\
b^T& B
\end{array}
\right)^{-1}  =
\left(\begin{array}{cc}
(a- b B^{-1} b^T)^{-1} &- a^{-1}b (B  - b^Ta^{-1}b)^{-1}\\
-B^{-1}b^T(a-bB^{-1}b^T)^{-1}& (B  - b^Ta^{-1}b)^{-1}
\end{array}
\right) ,
\]
where $a\in \R$ and $b , B$ are $1\times n$ and $n\times n$ matrices respectively.
This formula can be checked by multiplying the right-hand side
by $\left(\begin{array}{ll}
a & b\\
b^T& B
\end{array}
\right)$.
\end{proof}

\medskip
\noindent
{\bf Acknowledgements:}
 M.~del Pino and M. Musso have been  partly supported by grants
 Fondecyt  1160135,  1150066, Fondo Basal CMM. The  research  of R. Jerrard is partially supported by NSERC of Canada.

\end{document}